\documentclass[a4paper,11pt]{article}

\usepackage{amsmath,amssymb,epsfig,
color
}

\makeatletter
\@addtoreset{equation}{section}

\makeatother

\newcommand{\bnull}{{\boldsymbol 0}}
\newcommand{\bff}{{\bf f}}
\newcommand{\bw}{{\bf w}}

\newcommand{\bq}{{\bf q}}

\newcommand{\be}{{\bf e}}
\newcommand{\bu}{{\bf u}}
\newcommand{\btu}{\widetilde{\bf u}}
\newcommand{\tU}{\widetilde{U}}
\newcommand{\btv}{\widetilde{\bf v}}

\newcommand{\bp}{{\bf p}}

\newcommand{\bx}{{\bf x}}

\newcommand{\by}{{\bf y}}
\newcommand{\bv}{{\bf v}}

\newcommand{\bpsi}{{\boldsymbol \psi}}

\newcommand{\spann}{\operatorname{span}}
\newcommand{\diag}{\operatorname{diag}}
\newcommand{\blockdiag}{\operatorname{blockdiag}}

\newcommand{\Chol}{\operatorname{Chol}}

\newcommand{\n}{\nonumber}

\newcommand{\cL}{{\mathcal L}}

\newcommand{\cN}{{\mathcal N}}

\newcommand{\cR}{{\mathcal R}}

\makeatletter
\@addtoreset{equation}{section}

\makeatother

\newtheorem{theorem}{Theorem}[section]
\newtheorem{definition}[theorem]{Definition}
\newtheorem{lemma}[theorem]{Lemma}

\newtheorem{proposition}[theorem]{Proposition}
\newtheorem{corollary}[theorem]{Corollary}

\newenvironment{proof}{\par \vspace{0.3cm} \noindent{\sc Proof:} \ignorespaces}%
{\nolinebreak\hfill $\square$\par \medskip}

\def\R{ \mathbb {R}}
\def\C{ \mathbb {C}}

\def\cN{\mathcal{N}}

\begin{document}

\sloppy

\title{Converse theorem on a contraction metric \\ for a periodic orbit}

\author{Peter Giesl\thanks{Department of Mathematics,
         University of Sussex,
         Falmer BN1 9QH,
         United Kingdom
(p.a.giesl@sussex.ac.uk).}}

		\maketitle
		\begin{abstract}
			Contraction analysis uses a local criterion to prove the long-term behaviour of a dynamical system.
			A contraction metric is a Riemannian metric with respect to which the distance between adjacent solutions contracts. If adjacent solutions in all directions perpendicular to the flow are contracted, then there exists a unique periodic orbit, which is exponentially stable and we obtain a bound on the rate of exponential attraction.
			
			In this paper we study the converse question and show that, given an exponentially stable periodic orbit, a contraction metric exists on its basin of attraction and we can recover the bound on the rate of exponential attraction.
		\end{abstract}
		
			

%	\linenumbers

	 \section{Introduction}
	 
The stability and basin of attraction of periodic orbits is an important problem in many applications. Already the determination of a periodic orbit is a non-trivial task as it involves solving the differential equation. The classical definition of stability, as well as its study using a Lyapunov function require the knowledge of the position of the periodic orbit which in many applications can only be approximated. An alternative way to study the stability and basin of attraction is contraction analysis, which is a local criterion and does not require us to know the location of the periodic orbit.

%
%
%Motivated by theoretical and computational studies of contraction metrics, we seek to prove converse statements of the existence of such contraction metrics. Different contraction metric types have been considered, e.g. Finsler metric.
%We will prove that a Riemannian contraction metric exists.
%
%Moreover, we need to specify the kind of attractor; in our case we will assume that the attractor is an equilibrium or periodic orbit. While the sufficiency has already been established, we will consider the necessity, and we will study the case of a compact as well as unbounded set.

Throughout the paper we will study the autonomous ODE
\begin{eqnarray}
\dot{\bx}&=&\bff(\bx)\label{ODE}
\end{eqnarray}
where $\bff\in C^\sigma(\mathbb R^n,\mathbb R^n)$ with $\sigma\ge 1$. We denote the solution $\bx(t)$ with initial condition $\bx(0)=\bx_0$ by $S_t\bx_0=\bx(t)$ and  assume that it exist for all $t\ge 0$.

In the next definition we will define a contraction metric on $\mathbb R^n$.
Note that $M(\bx)$ defines a point-dependent scalar product through $\langle \bv,\bw\rangle_{M(\bx)}=\bv^TM(\bx)\bw$ for all $\bv,\bw\in\mathbb R^n$.

\begin{definition}[Contraction metric]\label{Riemannian}
	A Riemannian metric is a function $M\in C^0(G,\mathbb S^n)$, where $G\subset \mathbb R^n$ is open and $\mathbb S^n$ denotes the symmetric $n\times n$ matrices, $M(\bx)$ is positive definite for all $\bx\in G$ and  the orbital derivative of $M$ exists for all $\bx\in G$ and is continuous, i.e.
$$M'(\bx)=\frac{d}{dt} M(S_t\bx)\big|_{t=0}$$
exists and is continuous. A sufficient condition for the latter is that $M\in C^1(G,\mathbb S^n)$; then $M_{ij}'(\bx)=\nabla M_{ij}(\bx)\cdot \bff(\bx)$ for all $i,j\in\{1,\ldots,n\}$.

Define \begin{eqnarray}
L_M(\bx;\bv)&:=&\frac{1}{2}\bv^T\left(M(\bx)D\bff(\bx)+D\bff(\bx)^TM(\bx)+M'(\bx)\right)\bv. \label{LM}
\end{eqnarray}

The Riemannian metric $M$ is called {\bf  contraction metric in $K\subset  G$ with exponent $-\nu<0 $} if $L_M(\bx)\le -\nu$ for all $\bx\in K$, where
\begin{eqnarray}
L_M(\bx)&:=&\max_{\bv^TM(\bx)\bv=1,\bv^T M(\bx)\bff(\bx)=0}L_M(\bx;\bv).\label{L_M}
\end{eqnarray}
\end{definition}

The following theorem shows the implications of the existence of such a contraction metric on a certain set in the phase space.

\begin{theorem}
	Let $\varnothing\not= K \subset \mathbb R^n$ be
	a compact, connected and positively invariant set which contains no equilibrium.
	Let $M$ be a   contraction metric in $K$ with exponent $-\nu<0 $, see Definition \ref{Riemannian}.
 
Then there exists one and only one periodic orbit $\Omega\subset K$. This periodic orbit
is exponentially asymptotically stable, and the real parts of all Floquet exponents -- 
except the trivial one -- are less than or equal to $-\nu$.
 Moreover, the basin of attraction $A(\Omega)$ contains $K$.
 \end{theorem}

This theorem goes back to Borg \cite{borg} with $M(\bx)=I$, and has been extended to a general Riemannian metric \cite{stenstroem}. For more results on contraction analysis for a periodic orbit see \cite{hartman,hartmanbook,leonov3,leonov96}.

 Note that a similar result holds with an equilibrium if the contraction takes place in all directions $\bv$, i.e. if $L_M(\bx)\le -\nu$ in \eqref{L_M} is replaced by
 $\cL_M(\bx):=\max_{\bv^TM(\bx)\bv=1}L_M(\bx;\bv)\le -\nu$. For more references on contraction analysis see \cite{lohmiller}, and for the relation to Finsler-Lyapunov functions see \cite{finsler}.

 Note that $L_M(\bx)$ is a continuous with respect to $\bx$ and, as we will show in the paper, also locally Lipschitz-continuous. Due to the maximum, however, it is not differentiable in general.

In this paper we are interested in converse results, i.e. given an exponentially stable periodic orbit, does a Riemannian contraction metric as in Definition \ref{Riemannian} exist? 
\cite{lohmiller} gives a converse theorem, but here $M(t,\bx)$ depends on $t$ and will, in general, become unbounded as $t\to\infty$.  In \cite{giesl04} the existence of such a contraction metric was shown on a given compact subset of $A(\Omega)$, first on the periodic orbit, using Floquet theory, and then on $K$, using a Lyapunov function. The local construction, however, neglected the fact that the Floquet representation of solutions of the first variation equation along the periodic orbit is in general not real, but complex. We will show in this paper, that, by choosing the complex Floquet representation appropriately, the constructed Riemannian metric is real-valued, thus justifying the arguments in \cite{giesl04}. Moreover, we will show the existence of a Riemannian metric on the whole, possibly unbounded basin of attraction by using a new construction. The Riemannian metric will be arbitrarily close to the true rate of exponential attraction. Let us summarize the main result of the paper in the following theorem.

\begin{theorem}\label{main}
	Let $\Omega$ be an exponentially stable periodic orbit of $\dot{\bx}=\bff(\bx)$, let $-\nu$ be the largest real part of all its non-trivial Floquet exponents and $\bff\in C^\sigma(\mathbb R^n,\mathbb R^n)$ with $\sigma\ge 3$.
	
Then for all $\epsilon\in (0,\nu/2)$ there exists a contraction metric $M\in C^{\sigma-1}( A(\Omega), \mathbb S^{n})$ in $A(\Omega)$ as in Definition \ref{Riemannian}  with exponent $-\nu+\epsilon<0 $, i.e.
\begin{eqnarray}L_M(\bx)&=&\frac{1}{2}\max_{\bv^TM(\bx)\bv=1,\bv^T M(\bx)\bff(\bx)=0}\bv^T\left(M(\bx)D\bff(\bx)+D\bff(\bx)^TM(\bx)+M'(\bx)\right)\bv
	\nonumber\\
	&\le& -\nu+\epsilon\label{esti}
\end{eqnarray}
holds for all $\bx\in A(\Omega)$.
\end{theorem}

The metric is constructed in several steps: first on the periodic orbit, then in a neighborhood, and finally in the whole basin of attraction. In the proof, we define a projection of points $\bx$ in a neighborhood of the periodic orbit onto the periodic orbit, namely onto $\bp\in \Omega$, such that  to $(\bx-\bp )^TM(\bp)\bff(\bp)=0$. This is then used to synchronize the times of solutions through $\bx$ and $\bp$, and to define a time-dependent distance between these solutions, which decreases exponentially.

Let us compare our result with  converse theorems for a contraction metric for an equilibrium. In \cite{converse}, three converse theorems were obtained: Theorem 4.1 constructs a metric on a given compact subset of the basin of attraction (see \cite{giesl04} for the case of a periodic orbit), Theorem 4.2 constructs a metric on the whole basin of attraction (see this paper for the case of a periodic orbit), while Theorem 4.4 constructs a metric as solution of a linear matrix-valued PDE (see \cite{other} for the case of a periodic orbit). The latter construction is beneficial for its computation by solving the PDE, and it also constructs a smooth function; however, the exponential rate of attraction cannot be recovered, which is an advantage of the approach in this paper.

%This can be seen in the 1-dimensional setting: consider $\dot{\bx}=-x$.The solution of (11) in that paper is
%$\dot{M}=(-\beta+2)M$, i.e. $M(x,t)=e^{(-\beta+2)t}$. Hence, of $\beta>2$, then $M$ is not uniformly positively definite, if $\beta<2$, then $M$ is unbounded. If $\beta=2$, then $M$ is bounded, in particular $M=1$.

%In higher dimensions, consider $\dot{x}=-x$, $\dot{y}=-2y$. Then $M$ needs to mirror the exact rates of exponential attraction in the respective directions. Computationally, $t>0$ is difficult to handle, and unbounded $M$ as well. For autonomous equations, $M(\bx)$ seems to be more natural.

%\cite{lohmiller} p. 690: the equation for $\Theta$ for not always have a solution; consider an equilibrium with exponential rate of attraction not $-1$, then the equation has a singularity at the equilibrium:

%$\dot{x}=-2$, then the equation becomes $-2x\frac{\partial \Theta}{\partial x}=\theta $, i.e. $\theta^2=\frac{B}{x}$, which is not defined at $x=0$.

Let us give an overview over the paper: In Section \ref{Floquet} we prove a special Floquet normal form to ensure that the contraction metric that we later construct on the periodic orbit is real-valued. In Section \ref{conv} we prove the main result of the paper, Theorem \ref{main}, showing the existence of a Riemannian metric on the whole basin of attraction. The section also contains Corollary \ref{help_other}, defining a projection onto the periodic orbit and related estimates. In the appendix we prove that $L_M$ is locally Lipschitz-continuous.

\section{Floquet normal form}
\label{Floquet}

%
%\begin{remark}
%	Note that (\ref{L}) is equivalent to
%	$$M(\bx)D\bff(\bx)+D\bff(\bx)^TM(\bx)+M'(\bx) \prec-2\nu M(\bx)$$
%	where $A\prec B$ means $A-B$ is negative definite, where $A,B\in\mathbb S^n$.
%	CHECK???
%	
%	Under the above assumptions, the functions $\cL(\bx)$ and $L_M(\bx)$ are continuous with respect to $\bx$.?????
%	
%	(Locally) Lipschitz continuous here????
%\end{remark}

Before we consider the Floquet normal form, we will prove a lemma which calculates  $L_M(\bx)$ for the Riemannian metric $M(\bx)=e^{2V(\bx)}N(\bx)$.

\begin{lemma}\label{lem}
	Let $N\colon \mathbb R^n\to \mathbb S^n$ be a Riemannian metric and $V\colon \mathbb R^n\to \mathbb R$ a continuous and orbitally continuously differentiable function.
	
	Then $M(\bx)=e^{2V(\bx)}N(\bx)$ is a Riemannian metric and
	\begin{eqnarray*}
		L_M(\bx)&=&L_N(\bx)+V'(\bx).
	\end{eqnarray*}
\end{lemma}
\begin{proof}
	It is clear that $M(\bx)$ is a positive definite for all $\bx$ since $e^{2V(\bx)}>0$.
	We have
	\begin{eqnarray*}
		L_M(\bx;\bv)&=&
		\frac{1}{2}\bv^T\left(M(\bx)D\bff(\bx)+D\bff(\bx)^TM(\bx)+M'(\bx)\right)\bv\\
		&=&
		\frac{1}{2}\bv^T\bigg(e^{2V(\bx)}N(\bx)D\bff(\bx)+e^{2V(\bx)}D\bff(\bx)^TN(\bx)\\
		&&\hspace{1.2cm}+
		e^{2V(\bx)}(2V'(\bx)N(\bx)+N'(\bx))\bigg)\bv\\
		&=&
		\frac{1}{2}\bw^T\left(N(\bx)D\bff(\bx)+D\bff(\bx)^TN(\bx)+
		N'(\bx)\right)\bw
		+\bw^TN(\bx)\bw \, V'(\bx)
	\end{eqnarray*}
	with $\bw=e^{V(\bx)}\bv$, so $L_M(\bx;\bv)=L_N(\bx; \bw)+\bw^TN(\bx)\bw \, V'(\bx)$.
	Thus,
	\begin{eqnarray*}
	%	\cL_M(\bx)&=&\max_{\bv^TM(\bx)\bv=1}L_M(\bx;\bv)\\
%		&=&\max_{\bw^TN(\bx)\bw=1}\left[L_N(\bx;\bw)
%		+\bw^TN(\bx)\bw V'(\bx)\right]\\
%		&=&\cL_N(\bx)+ V'(\bx)\mbox{ and }\\
		L_M(\bx)&=&\max_{\bv^TM(\bx)\bv=1,\bv^TM(\bx)\bff(\bx)=0}L_M(\bx;\bv)\\
		&=&\max_{\bw^TN(\bx)\bw=1,\bw^T N(\bx)\bff(\bx)=0}\left[L_N(\bx;\bw)
		+\bw^TN(\bx)\bw V'(\bx)\right]\\
		&=&L_N(\bx)+ V'(\bx).
	\end{eqnarray*}
	This shows the lemma.
\end{proof} 

In order to show later that our constructed Riemannian metric $M$ is real-valued, we will construct a special Floquet normal form in Proposition \ref{prop} such that the matrix in \eqref{real} is real-valued. In Corollary \ref{coro} we will show estimates in the case that \eqref{var} is the first variation equation of a periodic orbit.
The proof of the following proposition is inspired by \cite{chicone}.

\begin{proposition}\label{prop}
	Consider the periodic differential equation
	\begin{eqnarray}\dot{\by}&=&F(t)\by\label{var}
	\end{eqnarray} where $F\in C^s(\mathbb R,\mathbb R^{n\times n})$ is $T$-periodic, $s\ge 1$ and denote by $\Phi\in C^s(\mathbb R,\mathbb R^{n\times n})$ its principal fundamental matrix solution with $\Phi(0)=I$. 
	
	Then there exists a $T$-periodic function $P\in C^s(\mathbb R,\mathbb C^{n\times n})$ with $P(0)=P(T)=I$ and a  matrix $B\in \mathbb C^{n\times n}$ such that for all $t\in\mathbb R$
	$$\Phi(t)=P(t)e^{Bt}.$$
	
	Denote by $\lambda_1,\ldots,\lambda_r\in \mathbb R\setminus \{0\}$ the pairwise distinct real eigenvalues and by $\lambda_{r+1},\overline{\lambda_{r+1}},\ldots,$
	$\lambda_{r+c},\overline{\lambda_{r+c}}\in\mathbb C\setminus \mathbb R$ the pairwise distinct pairs of complex conjugate complex eigenvalues of $\Phi(T)$ with algebraic multiplicity $m_j$ of $\lambda_j$.
	For $\epsilon>0$ there exists a non-singular matrix $S\in \mathbb R^{n\times n}$ such that $B=SAS^{-1}$ with $A=\blockdiag(K_1,K_2,\ldots,K_{r+c})$ and $K_j\in \mathbb C^{m_j\times m_j}$ for $j=1,\ldots,r$ and   $K_j\in \mathbb R^{2m_j\times 2m_j}$ for $j=r+1,\ldots,r+c$ as well as
	$$\frac{1}{2}\bw^*(A^*+A)\bw\le \sum_{j=1}^{r+c}c_j\sum_{i=1}^{m_j}|w_{i+\sum_{k=1}^{j-1}m_k}|^2
	\text{	for all }\bw\in \mathbb C^n,$$
	where $c_j=\left(\frac{\ln |\lambda_j|}{T}+\epsilon\right)$ if $m_j\ge 2$ and  $c_j=\frac{\ln |\lambda_j|}{T}$ if $m_j=1$.

 Moreover, we have
\begin{eqnarray}
(P^{-1}(t))^*(S^{-1})^*S^{-1}P^{-1}(t)&\in& \mathbb R^{n\times n}\label{real}
\end{eqnarray}
	for all $t\in \mathbb R$.
\end{proposition}
\begin{proof}
	Since $F\in C^s$, we also have $\Phi\in C^s(\mathbb R,\mathbb R^{n\times n})$. 
	Noting that $\Psi(t):=\Phi(t+T)$ solves \eqref{var} with $\Psi(0)=\Phi(T)$, we obtain from the uniqueness of solutions that 
	\begin{eqnarray}
		\Phi(t+T)&=&\Psi(t)=\Phi(t)\Phi(T)\text{ for all }t\in \mathbb R.\label{uni}
	\end{eqnarray}
	
	Consider $C:=\Phi(T)\in \mathbb R^{n\times n}$ which is non-singular and hence all eigenvalues of $\Phi(T)$ are non-zero. Let $\epsilon':=\frac{1}{2}\min\left( \frac{\epsilon T}{2},1\right)$ and $S\in \mathbb R^{n\times n}$ be such that
	$S^{-1}CS=:J$ is in real Jordan normal form with the $1$ replaced by $\epsilon'  |\lambda_j|$ for each eigenvalue $\lambda_j$, i.e. $J$ is a block-diagonal matrix with blocks $J_j$ of the form
	$J_j=\left(\begin{array}{lllll}\lambda_j&\epsilon' |\lambda_j|&&&\\
	&\lambda_j&\epsilon' |\lambda_j|&&\\
	&&\ddots&\ddots&\\
	&&&\lambda_j&\epsilon' |\lambda_j|\\
	&&&&\lambda_j\end{array}\right)\in\mathbb R^{m_j\times m_j}$ for real eigenvalues $\lambda_j$ of $C$ and
	$J_j=\left(\begin{array}{ccccccc}\alpha_j&-\beta_j&\epsilon' r_j&&&&\\
	\beta_j&\alpha_j&&\epsilon' r_j&&&\\
	&&\ddots&&\ddots&&\\
&&&\alpha_j&-\beta_j&\epsilon' r_j&\\
	&&&\beta_j&\alpha_j&&\epsilon' r_j\\
	&&&&&\alpha_j&-\beta_j\\
	&&&&&
	\beta_j&\alpha_j\end{array}\right)\in\mathbb R^{2m_j\times 2m_j}$ for each pair of complex eigenvalues $\alpha_j\pm i\beta_j$ of $C$, where $r_j=\sqrt{\alpha_j^2+\beta_j^2}$ and $m_j$ denotes the dimension of the generalized eigenspace of one of them; note we have pairs of complex conjugate eigenvalues since $C$ is real.
	
	This can be achieved by letting  $S_1\in\mathbb R^{n\times n}$ be an invertible matrix such that
	$S^{-1}_1CS_1$ is the standard real Jordan Normal Form with  $1$ on the super diagonal. Then define $S_2$ to be a matrix of blocks
	$$\diag(1,\epsilon' |\lambda_j|,
	(\epsilon')^2|\lambda_j|^2,\ldots,(\epsilon')^{m_j-1}|\lambda_j|^{m_j-1})$$
	for real $\lambda_j$ and 
	
	$$\diag(1,1,\epsilon' |\lambda_j|,\epsilon' |\lambda_j|, \ldots,(\epsilon')^{m_j-1}|\lambda_j|^{m_j-1},(\epsilon')^{m_j-1}|\lambda_j|^{m_j-1})$$ for a pair of complex conjugate eigenvalues $\lambda_j$ and $\overline{\lambda_j}$. Setting $S=S_1S_2$ yields the result.

	For each of the blocks, we will now construct a matrix $K_j\in\mathbb C^{m_j\times m_j}$ for real eigenvalues $\lambda_j$ and $K_j\in\mathbb R^{2m_j\times 2m_j}$ for each pair of complex eigenvalues $\alpha_j\pm i\beta_j$ such that
	$$e^{K_j T}=J_j,$$
	which shows with $B=SAS^{-1}$, where $A:=\blockdiag(K_1,\ldots,K_r)$,
\begin{eqnarray}
	e^{BT}&=&Se^{AT}S^{-1}=S\blockdiag(e^{K_1T},\ldots,e^{K_rT})S^{-1}\nonumber\\
	&=&SJS^{-1}=C=\Phi(T).\label{*}
	\end{eqnarray}
	
	We distinguish between three cases:  $\lambda_j$ being real positive, real negative or complex.
	Using the series expansion of $
	\ln (1+x)$ we obtain for a nilpotent matrix $M\in \mathbb R^{n\times n}$
\begin{eqnarray}
	\exp\left(\sum_{k=1}^\infty \frac{(-1)^{k+1}}{k}M^k\right)&=&I+M;\label{log}
	\end{eqnarray}
	note  that the sum is actually finite.
	
	\vspace{0.3cm}
	
\noindent	\underline{\bf Case 1: $\lambda_j\in \mathbb R^+$}
	
	Writing $J_j=\lambda_j(I+\epsilon'  N)$  with the nilpotent matrix $N=
	\left(\begin{array}{llll}0&1&&\\
	&\ddots&\ddots&\\
	&&0&1\\
	&&&0\end{array}\right)\in\mathbb R^{m_j\times m_j}$, we define 
	$$K_j=\frac{1}{T}\left( (\ln \lambda_j)I+\sum_{k=1}^{m_j-1} \frac{(-1)^{k+1}}{k} (\epsilon')^kN^k\right)\in \mathbb R^{m_j\times m_j}.$$
Since $I$ and $N$ commute, we have with \eqref{log} and $N^k=0$ for $k\ge m_j$
	\begin{eqnarray*}
		\exp(K_j T)=\lambda_j \left(I+\epsilon' N\right)=J_j.
	\end{eqnarray*}

\vspace{0.5cm}
%\newpage

\noindent	\underline{\bf Case 2: $\lambda_j\in \mathbb R^-$}
	
 With the nilpotent matrix $N=
 \left(\begin{array}{llll}0&1&&\\
 &\ddots&\ddots&\\
 &&0&1\\
 &&&0\end{array}\right)\in\mathbb R^{m_j\times m_j}$ we	write $J_j=-|\lambda_j| (I-\epsilon' N)$ and define 
	$$K_j=\frac{1}{T}\left( (i\pi +\ln |\lambda_j|)I+\sum_{k=1}^{m_j-1}  \frac{(-1)^{k+1}}{k}(-\epsilon')^kN^k\right)\in \mathbb C^{m_j\times m_j}.$$
	Since $I$ and $N$ commute, and $N^k=0$ for $k\ge m_j$ we have with \eqref{log}
	\begin{eqnarray*}
		\exp(K_j T)=-|\lambda_j | \left(I-\epsilon' N\right)=J_j.
	\end{eqnarray*}

	\vspace{0.3cm}
	
	\noindent
	\underline{\bf Case 3: $\lambda_j=\alpha_j+i\beta_j$ with $\beta_j\not=0$ }
	
	We only consider one of the two complex conjugate eigenvalues $\lambda_j$ and $\overline{\lambda_j}$ of $\Phi(T)$.
		Writing $\lambda_j$ in polar coordinates gives
	$\lambda_j=\alpha_j+i\beta_j=r_je^{i\theta_j}=r_j\cos\theta_j + i r_j\sin \theta_j$ with $r_j>0$ and $\theta_j \in (0,2\pi)$. Then,
	defining $R_j=r_j\left(\begin{array}{cc}\cos \theta_j &-\sin  \theta_j\\
	\sin \theta_j&\cos \theta_j\end{array}\right)$,  $\cR=\blockdiag (R_j,R_j,\ldots,R_j)\in \mathbb R^{2m_j\times 2m_j}$ and the nilpotent matrix $\cN
	\in\mathbb R^{2m_j\times 2m_j}$ having $2\times 2$ blocks of $ \left(\begin{array}{cc}\cos \theta_j &\sin  \theta_j\\
	-\sin \theta_j&\cos \theta_j\end{array}\right)=\left(\begin{array}{cc}\cos \theta_j &-\sin  \theta_j\\
	\sin \theta_j&\cos \theta_j\end{array}\right)^{-1}$ above its diagonal, we have
	$J_j=\cR(I+\epsilon' \cN)$.	
	We define $\Theta=\left(\begin{array}{cc}0 &-\theta_j\\
	\theta_j&0\end{array}\right)$ and
	$$K_j=\frac{1}{T}\left( (\ln r_j)I+\blockdiag(\Theta,\Theta,\ldots,\Theta)+ \sum_{k=1}^{2m_j-2}\frac{(-1)^{k+1}}{k} (\epsilon') ^k\cN^k\right)\in \mathbb R^{2m_j\times 2m_j}.$$
%	note that the sum is finite since $\cN$ is nilpotent.
Since $I$, $\blockdiag(\Theta,\Theta,\ldots,\Theta)$ and $\cN$ commute, we have, using $\cN^k=0$ for $k\ge 2m_j-1$ and \eqref{log}
	\begin{eqnarray*}
		\lefteqn{
			\exp(K_j T)}\\&=&r_j \blockdiag\left( \left(\begin{array}{cc}\cos \theta_j &-\sin  \theta_j\\
			\sin \theta_j&\cos \theta_j\end{array}\right),\ldots,\left(\begin{array}{cc}\cos \theta_j &-\sin  \theta_j\\
			\sin \theta_j&\cos \theta_j\end{array}\right)\right) (I+\epsilon'\cN)\\
		&=&J_j.
	\end{eqnarray*}
	
	We can now define $P\in C^s(\mathbb R, \mathbb C^{n\times n})$ by $P(t)=\Phi(t)e^{-Bt}$, which satisfies $P(0)=I$ and 
	\begin{eqnarray*}
		P(t+T)&=&\Phi(t+T)e^{-BT}e^{-Bt}\\
		&=&\Phi(t)\Phi(T)e^{-BT}e^{-Bt}\text{ by }\eqref{uni}\\
		&=&P(t)\text{ by }\eqref{*}
		\end{eqnarray*}
		for all $t\ge 0$, so in particular $P(T)=P(0)=I$. We can now write
			$$\Phi(t)=P(t)e^{Bt}.$$
			This shows the first statement of the proposition.
			
\vspace{0.3cm}	
	We now evaluate $A^*+A=\blockdiag(K_1^*+K_1,\ldots,K_r^*+K_r)$. Let us consider $K_j$ as in the three cases above. If $m_j=1$, then $K_j$ below does not contain the last sum with $\epsilon'$ and the form of $c_j$ is immediately clear.

	\vspace{0.3cm}
	
	\noindent
	\underline{\bf Case 1: $\lambda_j\in \mathbb R^+$}
	
	$$K_j=\frac{1}{T}\left( (\ln \lambda_j)I+\sum_{k=1}^{m_j-1} \frac{(-1)^{k+1}}{k}(\epsilon')^kN^k\right) \in\mathbb R^{m_j\times m_j};$$
	hence, for $\bw\in \mathbb C^{m_j}$
	\begin{eqnarray*}
		\lefteqn{
		\frac{1}{2}\bw^*(K_j^*+K_j)\bw}\\
		&=& 
		\frac{\ln \lambda_j}{T}\sum_{i=1}^{m_j}|w_i|^2\\
		&&+\epsilon'
		\frac{1}{2T}
		\left(\overline{w_1}w_2+w_1\overline{w_2}+\overline{w_2}w_3+w_2\overline{w_3}+\ldots+
		\overline{w_{m_j-1}}w_{m_j}+w_{m_j-1}\overline{w_{m_j}}\right)\\
		&&
		-\frac{(\epsilon')^2}{2}
		\frac{1}{2T}
		\left(\overline{w_1}w_3+w_1\overline{w_3}+\overline{w_2}w_4+w_2\overline{w_4}+\ldots+
		\overline{w_{m_j-2}}w_{m_j}+w_{m_j-2}\overline{w_{m_j}}\right)\\
		&&+\ldots\\
		&&
		+(-1)^{m_j}\frac{(\epsilon')^{m_j-1}}{m_j-1}
		\frac{1}{2T}
		\left(\overline{w_1}w_{m_j}+w_1 \overline{w_{m_j}}\right)\,.
	\end{eqnarray*}
	
	Note that the Cauchy--Schwarz inequality implies $\R\ni\overline{\xi}\eta+\xi\overline{\eta}\le |\xi|^2+|\eta|^2$, which yields that, using $\epsilon'=\frac{1}{2}\min\left(\frac{\epsilon T}{2},1\right)$ 
	\begin{eqnarray*}
		\frac{1}{2}\bw^*(K_j^*+K_j)\bw
		&\le &		\frac{\ln \lambda_j}{T}\sum_{i=1}^{m_j}|w_i|^2\\
		&&+\frac{\epsilon'+( \epsilon')^2 + \ldots+(\epsilon')^{m_j-1} }{T}\sum_{i=1}^{m_j}|w_i|^2\\
		&\le &
		\left(
		\frac{\ln \lambda_j}{T} +\epsilon\left(\frac{1}{2}+\frac{1}{4}+\frac{1}{8}+\ldots \right)\right)\sum_{i=1}^{m_j}|w_i|^2\\
		&\le&
		\left(
		\frac{\ln \lambda_j}{T} +\epsilon\right)\sum_{i=1}^{m_j}|w_i|^2.
	\end{eqnarray*}

	\vspace{0.3cm}
	\newpage
	
	\noindent
	\underline{\bf Case 2: $\lambda_j\in \mathbb R^-$}

	$$K_j=\frac{1}{T}\left( (i\pi +\ln |\lambda_j|)I+\sum_{k=1}^{m_j-1}  \frac{(-1)^{k+1}}{k}(-\epsilon')^kN^k\right)\in \mathbb C^{m_j\times m_j};$$
	hence, for $\bw\in \mathbb C^{m_j}$
	\begin{eqnarray*}
		\lefteqn{
		\frac{1}{2}\bw^*(K_j^*+K_j)\bw}\\
		&=& 
		\frac{\ln |\lambda_j|}{T}\sum_{i=1}^{m_j}|w_i|^2\\
		&&+\epsilon'
		\frac{1}{2T}
		\left(\overline{w_1}w_2+w_1\overline{w_2}+\overline{w_2}w_3+w_2\overline{w_3}+\ldots+
		\overline{w_{m_j-1}}w_{m_j}+w_{m_j-1}\overline{w_{m_j}}\right)\\
		&&
		-\frac{(\epsilon')^2}{2}
		\frac{1}{2T}
		\left(\overline{w_1}w_3+w_1\overline{w_3}+\overline{w_2}w_4+w_2\overline{w_4}+\ldots+
		\overline{w_{m_j-2}}w_{m_j}+w_{m_j-2}\overline{w_{m_j}}\right)\\
		&&+\ldots\\
		&&
		+(-1)^{m_j}\frac{(\epsilon')^{m_j-1}}{m_j-1}
		\frac{1}{2T}
		\left(\overline{w_1}w_{m_j}+w_1 \overline{w_{m_j}}\right)\\
		&\le&
		\left(
		\frac{\ln |\lambda_j|}{T} +\epsilon\right)\sum_{i=1}^{m_j}|w_i|^2
	\end{eqnarray*}
	similarly to case 1.

	\vspace{0.3cm}
	
	\noindent
	\underline{\bf Case 3: $\lambda_j=\alpha_j+i\beta_j$ with $\beta_j\not=0$ }
	
	Recall that 
	$$K_j=\frac{1}{T}\left( (\ln r_j)I+\blockdiag(\Theta,\Theta,\ldots,\Theta)+ \sum_{k=1}^{2m_j-2} \frac{(-1)^{k+1}}{k}(\epsilon' )^k\cN^k\right)\in \mathbb R^{2m_j\times 2m_j};$$
	where $\Theta=\left(\begin{array}{cc}0 &-\theta_j\\
	\theta_j&0\end{array}\right)$ and  the nilpotent matrix $\cN$ has $2\times 2$ blocks of $ \left(\begin{array}{cc}\cos \theta_j &\sin  \theta_j\\
	-\sin \theta_j&\cos \theta_j\end{array}\right) $ on its super diagonal.
	Note that all entries of $\cN^k$, $k\in \mathbb N$ are real and have an absolute value of $\le 1$ as they are of the form 
	$\cos(k\theta_j)$ and $\pm \sin(k\theta_j)$ for $k=1,2,\ldots$.
	Hence, for $\bw\in \mathbb C^{2m_j}$
		\begin{eqnarray*}
		\frac{1}{2}\bw^*(K_j^*+K_j)\bw
		&=& 
		\frac{\ln r_j}{T}\sum_{i=1}^{2m_j}|w_i|^2\\
		&&+\epsilon'
		\frac{1}{2T}
		\bigg(\cos\theta_j (\overline{w_1}w_3+w_1\overline{w_3})+\sin\theta_j (\overline{w_1}w_4+w_1\overline{w_4})\\
		&&
		-\sin\theta_j (\overline{w_2}w_3+w_2\overline{w_3})+\cos\theta_j (\overline{w_2}w_4+w_2\overline{w_4})+\ldots
		\bigg)+\ldots
		\end{eqnarray*}
		\begin{eqnarray*}
		&\le& 
		\frac{\ln r_j}{T}\sum_{i=1}^{2m_j}|w_i|^2\\
		&&+2
		\frac{\epsilon'+(\epsilon')^2+\ldots+(\epsilon')^{2m_j-1}}{T}\sum_{i=1}^{2m_j}|w_i|^2
		\\
		&\le &
		\left(
		\frac{\ln r_j}{T} +\epsilon\left(\frac{1}{2}+\frac{1}{4}+\frac{1}{8}+\ldots \right)\right)\sum_{i=1}^{2m_j}|w_i|^2\\
		&\le&
		\left(
		\frac{\ln r_j}{T} +\epsilon\right)\sum_{i=1}^{2m_j}|w_i|^2
	\end{eqnarray*}
	since $\epsilon'= \min \left(\frac{\epsilon T}{2},1\right)$.
	This shows the second statement of the proposition.

	To show that $(P^{-1}(t))^*(S^{-1})^*S^{-1}P^{-1}(t)$ has real entries, note that
	\begin{eqnarray*}
		P^{-1}(t)&=&e^{Bt}\Phi^{-1}(t)\\
		&=&Se^{At}S^{-1}\Phi^{-1}(t)
		\end{eqnarray*}
		so that
	\begin{eqnarray*} 
(P^{-1}(t))^*(S^{-1})^*S^{-1}P^{-1}(t)
&=&(\Phi^{-1}(t))^*(S^{-1})^*
(e^{At})^*e^{At}S^{-1}\Phi^{-1}(t).		
\end{eqnarray*}
It is thus sufficient to show that $(e^{At})^*e^{At}$ is real-valued, since all other matrices are real-valued.
Note that since 
	$A=\blockdiag(K_1,\ldots,K_r)$, we have
	\begin{eqnarray*}
		e^{At}&=&\blockdiag(e^{tK_1},\ldots,e^{tK_r}),\\
		(e^{tA})^*
		e^{tA}&=&\blockdiag((e^{tK_1})^*e^{tK_1},\ldots,(e^{tK_r})^*e^{tK_r})
	\end{eqnarray*} and the 
	blocks where $K_j$ have only real entries are trivially real-valued (cases 1 and 3).
	In case 2,   $K_j=\frac{1}{T}((i\pi+ \ln |\lambda_j|)I+N')$, where $N'\in\mathbb R^{m_j\times m_j}$ is a nilpotent, upper triangular matrix. Then, noting that $I$ and $N'$ commute,
	\begin{eqnarray*}
		e^{tK_j}&=&e^{\frac{  t}{T}(i\pi+\ln |\lambda_j|)} \exp\left(\frac{t}{T}N'\right)\\
		(e^{tK_j})^*&=&e^{\frac{  t}{T}(-i\pi+\ln |\lambda_j|)} \exp\left(\frac{t}{T}(N')^T\right)\\
		(e^{tK_j})^*e^{tK_j}&=&e^{\frac{  2t}{T}\ln |\lambda_j|}\exp\left(\frac{t}{T}(N')^T\right)\exp\left(\frac{t}{T}N'\right),
	\end{eqnarray*}
	which has real entries.
\end{proof}

\begin{corollary}\label{coro}
	Consider the ODE $\dot{\bx}=\bff(\bx)$ with $f\in C^\sigma(\mathbb R^n,\mathbb R^n)$, $\sigma\ge 2$ and let $S_t\bq$ be an exponentially stable periodic solution with period $T$ and $\bq\in \mathbb R^n$.
	 Then the first variation equation $\dot{\by}=D\bff(S_t \bq)\by$ is  of the form
	as in the previous proposition with $s=\sigma-1$; 1 is a single eigenvalue of $\Phi(T)$ with eigenvector $\bff(\bq)$ and all other eigenvalues of $\Phi(T)$ satisfy $|\lambda|<1$.
	More precisely, if $-\nu<0$ is the maximal real part of all non-trivial Floquet exponents, we have $\frac{\ln |\lambda|}{T}\le -\nu$.
	With the notations of Proposition \ref{prop} we can assume that $\lambda_1=1$ and $S\be_1=\bff(\bq)$. 
	
	Then we have for all $\epsilon> 0$
	\begin{eqnarray*}
		\bff(S_t\bq)&=&P(t)S\be_1\text{ for all }t\in\mathbb R\\
	\text{ and }\frac{1}{2}\bw^*(A^*+A)\bw&\le& \left(-\nu+\epsilon\right)(\|\bw\|^2-|w_1|^2).
	\end{eqnarray*}
	for all $\bw\in \mathbb C^n$, where $\|\bw\|=\sqrt{\bw^* \bw}$. 
\end{corollary}
\begin{proof}
	Since $\bff(S_t\bq)$ solves \eqref{var}, we have $\bff(S_t\bq)=P(t)e^{Bt}\bff(\bq)$ and, in particular for $t=T$,  $\bff(\bq)=\bff(S_T\bq)=e^{BT}\bff(\bq)$.
	Hence, \begin{eqnarray*}
		\bff(S_t\bq)&=&P(t)
Se^{At}S^{-1}\bff(\bq)\\
&=&P(t)Se^{At}\be_1\\
&=&P(t)	S\be_1
\end{eqnarray*}
since $K_1=0$ in the definition of $A$.
Proposition \ref{prop}  shows the result taking $\lambda_1=1$ and $m_1=1$ into account.
\end{proof}

\section{Converse theorem}\label{conv}

We will prove Theorem \ref{main}, showing that a contraction metric exists for an exponentially stable periodic orbit in the whole basin of attraction. Moreover, we can achieve the bound  $-\nu+\epsilon$ for $L_M$  for any fixed $\epsilon>0$, where $-\nu$ denotes the largest real part of all non-trivial Floquet exponents. 

Note that we consider contraction in directions $\bv$ perpendicular to $\bff(\bx)$ with respect to the  metric $M$, i.e. $\bv^TM(\bx)\bff(\bx)=0$. One could alternatively consider directions perpendicular to $\bff(\bx)$ with respect to the Euclidean metric, i.e. $\bv^T\bff(\bx)=0$, but then the function $L_M$ needs to reflect this, see \cite{other,leonov1}.

In the proof we will first construct $M=M_0$ on the periodic orbit $\Omega$ using Floquet theory. Then, we define a projection $\pi$ of points in a neighborhood $U$ of $\Omega$ onto $\Omega$ such that $(\bx-\pi(\bx))^TM_0(\pi(\bx))\bff(\pi(\bx))=0$, which will be used to synchronize the time of solutions such that
$\pi(S_\tau \bx)=S_{\theta_\bx(\tau)}\pi(\bx)$. Finally, $M$ will be defined through a scalar-valued function $V$ by
$M(\bx)=M_1(\bx)e^{2V(\bx)}$, where $M_1=M_0$ on the periodic orbit.

\begin{proof} [of Theorem \ref{main}]
Note that we assume $\bff\in C^\sigma(\mathbb R^n,\mathbb R^n)$ to achieve more detailed results concerning the smoothness and assume lower bounds on $\sigma$ as appropriate for each result; we always assume $\sigma\ge 2$.
%; however, for the statement of the theorem $\sigma=2$ is sufficient.

\vspace{0.3cm}
\newpage
\noindent \underline{\bf I. Definition and properties of $M_0$ on $\Omega$}

\noindent We fix a point $\bq\in\Omega$ and consider the first variation equation
\begin{eqnarray}
\dot{\by}&=&D\bff(S_t\bq)\by\label{variation}
\end{eqnarray}
which is a $T$-periodic, linear equation for $\by$, and $D\bff\in C^{\sigma-1}$. By Proposition \ref{prop} and Corollary \ref{coro}  the principal fundamental matrix  solution $\Phi\in C^{\sigma-1}(\mathbb R,\mathbb R^{n\times n})$ of (\ref{variation}) with $\Phi(0)=I$ can be written as
$$\Phi(t)=P(S_t\bq)e^{Bt},$$
where $B\in\mathbb C^{n\times n}$;
note that $P\in C^{\sigma-1}(\mathbb R,  \mathbb C^{n\times n})$ can be defined on the periodic orbit as it is $T$-periodic. By the assumptions on $\Omega$, the eigenvalues of $B$ are $0$ with algebraic multiplicity one and the others have a real part  $\le -\nu<0$.

  We define $S$ as in Proposition \ref{prop} and define the $C^{\sigma-1}  $-function
  \begin{eqnarray}
M_0(S_t \bq)&=&P^{-1}(S_t\bq)^*(S^{-1})^*S^{-1}P^{-1}(S_t\bq)\in\mathbb R^{n\times n}.\label{M_0}
\end{eqnarray}
Note that $M_0(S_t\bq)$ is real by Proposition \ref{prop}, symmetric, since it is Hermitian and real, and positive definite  by
\begin{eqnarray}
	\bv^TM_0(S_t\bq)\bv	&=&\|S^{-1}P^{-1}(S_t\bq)\bv\|^2 \mbox{ for all } \bv\in\R^n\label{N=M}
\end{eqnarray}
and since $S^{-1}P^{-1}(S_t\bq)$ is non-singular.

  We will now calculate
   $L_{M_0}(S_t\bq;\bv)$.    First, we have for the orbital derivative 
     $$
  M_0'(S_t\bq)=(({P}^{-1}(S_t\bq))')^\ast (S^{-1})^\ast S^{-1}P^{-1}(S_t\bq) +P^{-1}(S_t\bq)^\ast (S^{-1})^\ast S^{-1}(P^{-1}(S_t\bq))'\,.
$$ 

Furthermore, by using $(P^{-1}(S_t\bq)P(S_t\bq ))'=0$, we obtain $$(P^{-1}(S_t\bq))'=-P^{-1}(S_t\bq)(P(S_t\bq))' P^{-1}(S_t\bq).$$ In addition, since $t\mapsto P(S_t\bq )e^{Bt }$ is a solution of \eqref{variation}, we have
  $(P(S_t\bq))'=D\bff(S_t\bq)P(S_t\bq)- P(S_t\bq )B$. Altogether, we get
  \begin{eqnarray}
 (P^{-1}(S_t\bq ))'&=&-P^{-1}(S_t\bq)D\bff(S_t\bq) +B P^{-1}(S_t\bq)\,.\label{alto}
  \end{eqnarray}
  Hence,
  \begin{align*}
  M_0'(S_t\bq) &=-D\bff(S_t\bq)^TP^{-1}(S_t\bq)^\ast  (S^{-1})^\ast S^{-1}P^{-1}(S_t\bq)\\
   &\quad \,+P^{-1}(S_t\bq)^\ast B^\ast (S^{-1})^\ast S^{-1}P^{-1}(S_t\bq)
    \\
    &\quad\, -P^{-1}(S_t\bq)^\ast (S^{-1})^\ast S^{-1}P^{-1}(S_t\bq)D\bff(S_t\bq)   \\
    &\quad\, +P^{-1}(S_t\bq)^\ast (S^{-1})^\ast S^{-1}B P^{-1}(S_t\bq).
  \end{align*}
  
  Thus, we obtain
  \begin{eqnarray*}
    \lefteqn{M_0(S_t\bq) D\bff(S_t\bq)+ D\bff(S_t\bq)^TM_0(S_t\bq)+ M_0'(S_t\bq)}
    \\
    &= &    P^{-1}(S_t\bq)^\ast B^\ast (S^{-1})^\ast S^{-1}P^{-1}(S_t\bq) +P^{-1}(S_t\bq)^\ast (S^{-1})^\ast S^{-1}B P^{-1}(S_t\bq)\,.
  \end{eqnarray*}
  Furthermore, we have for $\bv\in \R^n$
  %, using $\bv^T Z \bv= (\bv^T Z \bv)^T=\bv^T Z^T \bv$ repeatedly, that
  \begin{eqnarray}
  L_{M_0}(S_t\bq;\bv)&=&
  \frac{1}{2}  \lefteqn{ \bv^T\left(M_0(S_t\bq) D\bff(S_t\bq)+ D\bff(S_t\bq)^TM_0(S_t\bq)+ M_0'(S_t\bq)\right)\bv}\nonumber
    \\
     &=& \bv^T P^{-1}(S_t\bq)^\ast (S^{-1})^\ast \left( \frac{1}{2}\left(S^\ast B^\ast (S^{-1})^\ast+S^{-1}BS\right)\right) S^{-1}P^{-1}(S_t\bq)\bv\nonumber
    \\
    &=&\bw^\ast\left(\frac{1}{2}\left(A^\ast+A\right)\right)\bw\,,\label{**}
  \end{eqnarray}
  where $\bw:=S^{-1}P^{-1}(S_t\bq)\bv\in \C^n$ and $A=S^{-1}BS$.
  
  For $\bv\in\mathbb R^n$ with $\bv^TM_0(S_t\bq)\bv=1$ and $\bff(S_t\bq)^TM_0(S_t\bq)\bv=0$  we have
  \begin{eqnarray*}
  \bw^*\bw&=&\bv^T P^{-1}(S_t\bq)^* (S^{-1})^*
 S^{-1}P^{-1}(S_t\bq)\bv\\
  &=&\bv^TM_0(S_t\bq)\bv\\
  &=&1
  \end{eqnarray*}
  and, using $\be_1=S^{-1}P^{-1}(S_t\bq)\bff(S_t\bq)$ from Corollary \ref{coro} \begin{eqnarray*}
  w_1&=&\be_1^*  \bw\\
  &=&\bff(S_t\bq)^T P^{-1}(S_t\bq)^* (S^{-1})^*
  S^{-1}P^{-1}(S_t\bq)\bv\\
  &=&\bff(S_t\bq)^TM_0(S_t\bq)\bv\\
  &=&0.
  \end{eqnarray*}
%Hence, for these $\bv$ we conclude from Corollary \ref{coro}
% \begin{eqnarray}
% 	L_{M_0}(S_t\bq;\bv)
% 	&\le& (-\nu+\epsilon')\|\bw\|^2\label{esti0}
% \end{eqnarray}
% where $\bw=S^{-1}P^{-1}(S_t\bq)\bv$.

This shows with Corollary \ref{coro} and \eqref{**}
\begin{eqnarray}
	L_{M_0}(S_t\bq)&=&
	\max_{\bv^TM_0(S_t\bq)\bv=1,\bv^TM_0(S_t\bq)\bff(S_t\bq)=0}
	L_{M_0}(S_t \bq;\bv)\n \\
	&\le&
	\max_{\bw\in \mathbb C^n,w_1=0,\|\bw\|=1}
	(-\nu+\epsilon)(\|\bw\|^2-|w_1|^2)\n\\
	& \le&-\nu+\epsilon.\label{LMv}
\end{eqnarray}

\vspace{0.3cm}
\noindent \underline{\bf II. Projection}

\noindent
Fix a point $\bq\in \Omega$ on the periodic orbit.
For $\bx$ near the periodic orbit we define the projection
$\pi({\bf x})=S_\theta \bq$ on the periodic orbit orthogonal
to $\bff(S_\theta \bq)$ with respect to the scalar product $\langle \bv,\bw\rangle_{M_0(S_\theta \bq)}=\bv^T M_0(S_\theta \bq)\bw$ implicitly by \eqref{proj} below.	
The following lemma is based on the implicit function theorem and shows that the projection can be defined in a neighborhood of the periodic orbit, not just locally.

\begin{lemma}\label{imp}
Let $\Omega$ be an exponentially stable periodic orbit of $\dot{\bx}=\bff(\bx)$ where $\bff\in C^\sigma( \mathbb R^n,\mathbb R^n)$ with $\sigma\ge 2$.

Then there is a compact, positively invariant neighborhood $U$ of $\Omega$ with $U\subset A(\Omega)$ and a function $\pi\in C^{\sigma-1}(U,\Omega)$ such that $\pi(\bx)=\bx$ if and only if $\bx\in \Omega$. Moreover, for all $\bx\in U$ we have
\begin{eqnarray}
	( {\bf x}-\pi(\bx))^TM_0(\pi(\bx))\bff(\pi(\bx))&=&0.\label{proj}
	\end{eqnarray}
\end{lemma}
\begin{proof} Fix a point $\bq\in \Omega$ and define $M_0$ by \eqref{M_0}.
Define the $C^{\sigma-1}$ function 
$$G(\bx,\theta)=(\bx-S_\theta \bq)^TM_0(S_\theta \bq)\bff(S_\theta \bq)$$
 for $\bx\in \mathbb R^n$, $\theta\in \mathbb R$.

Define the following constants:
\begin{eqnarray*}
	\min_{\bp\in \Omega}\|\bff(\bp)\|&= & c_1>0\\
	\max_{\bp\in \Omega}\|\bff(\bp)\|&= & c_2\\
	\max_{\bp\in \Omega}\|D\bff(\bp)\|&=&c_3\\
	\max_{\bp\in \Omega}\|P(\bp)\|&=&p_1\\
	\max_{\bp\in \Omega}\|P^{-1}(\bp)\|&=&p_2\\
	\min_{\bp\in \Omega}\|M_0(\bp)\|&= & m_1>0\\
	\max_{\bp\in \Omega}\|M_0(\bp)\|&= & m_2\\
	\max_{\bp\in \Omega}\|M_0'(\bp)\|&= & m_3,
\end{eqnarray*}
with the matrix norm $\|\cdot\|=\|\cdot\|_2$, which is induced by the vector norm $\|\cdot\|=\|\cdot\|_2$ and is sub-multiplicative. 
We will first prove the following quantitative version of the local implicit function theorem, using that $\theta$ is one-dimensional.

\begin{lemma}\label{implicit}
There are constants $\delta,\epsilon>0$ such that 
for each point $\bx_0=S_{\theta_0}\bq \in \Omega$, there is a function $p_{\bx_0}\in C^{\sigma-1}( B_\delta(\bx_0), B_\epsilon(\theta_0))$ such that 
  for all $(\bx,\theta)\in B_\delta(\bx_0)\times B_\epsilon(\theta_0)$
$$G(\bx,\theta)=0\Longleftrightarrow \theta=p_{\bx_0}(\bx).$$
%Note that the constant $\epsilon$ and $\delta$ are uniform for all points $\bx_0=S_{\theta_0}\bq \in\Omega$.
If $\bx\in B_{\delta/2}(\bx_0)$, then $p_{\bx_0}(\bx)\in B_{\epsilon/2}(\theta_0).$
\end{lemma}
\begin{proof}
We have 
\begin{eqnarray*}
G_\theta(\bx,\theta)&=&\frac{d}{d\theta}( {\bf x}-S_\theta \bq)^TM_0(S_\theta \bq)\bff(S_\theta \bq)\\
&=&-\bff(S_\theta \bq)^TM_0(S_\theta\bq)\bff(S_\theta \bq)\\
&&+
( {\bf x}-S_\theta \bq)^TM_0'(S_\theta \bq)\bff(S_\theta \bq)\\
&&
+( {\bf x}-S_\theta \bq)^TM_0(S_\theta \bq)D\bff(S_\theta
\bq)\bff(S_\theta \bq)\,.
\end{eqnarray*}
With $\min_{\theta \in
[0,T]}\bff(S_\theta \bq)^TM_0(S_\theta\bq)\bff(S_\theta \bq)\ge c_1^2m_1>0$ we have for all 
$ \|{\bf x}-S_\theta \bq\|<\delta_2:= \frac{c_1^2m_1}{2c_2 (m_3+m_1c_3)}$
\begin{eqnarray*}
G_\theta(\bx,\theta)&<& -c_1^2m_1
+\delta_2c_2 (m_3+m_1c_3)\
= \ -\frac{c_1^2m_1}{2}<0. 
\end{eqnarray*}

Let $\delta_1:=\frac{\delta_2}{2} $ and $\epsilon_1:=\frac{\delta_2}{2c_2}$.
For 
 any $\bx_0=S_{\theta_0}\bq\in \Omega$ we have for all  $\bx\in \mathbb R^n$ with
$\|\bx-\bx_0\|<\delta_1$ and all $\theta\in \mathbb R$ with  $|\theta-\theta_0|<\epsilon_1$
\begin{eqnarray}
G_\theta(\bx,\theta)&<& -\frac{c_1^2m_1}{2}<0 \label{Ftheta}
\end{eqnarray} since
$$\|\bx-S_\theta\bq\|\le \|\bx-\bx_0\|+\|S_{\theta_0}\bq-S_\theta \bq\|<
\delta_1+|\theta_0-\theta| c_2<\delta_2.$$

Since $G(\bx_0,\theta_0)=0$ we have with $\epsilon:=\epsilon_1/2$ by \eqref{Ftheta}  \begin{eqnarray*}
G(\bx_0,\theta_0+\epsilon)&<&-\frac{c_1^2m_1}{2}\epsilon,\\
G(\bx_0,\theta_0-\epsilon)&>&\frac{c_1^2m_1}{2}\epsilon.
\end{eqnarray*}

Furthermore, we have
\begin{eqnarray*}
\nabla_\bx G(\bx,\theta)&=&\bff(S_\theta \bq)^TM_0(S_\theta \bq),\\
\|\nabla_\bx G(\bx,\theta)\|&\le&c_2m_1
\end{eqnarray*}
for all $\bx\in \mathbb R^n$ and $\theta\in \mathbb R$.

Define $\delta:=\min\left(\delta_1,\frac{c_1^2m_1}{4c_2m_1}\epsilon\right)$. 
For $\|\bx-\bx_0\|< \delta$ we have
\begin{eqnarray*}G(\bx,\theta_0+\epsilon)&\le&G(\bx_0,\theta_0+\epsilon)\\
	&&+\int_0^1
\nabla_\bx G(\bx_0+\lambda (\bx-\bx_0),\theta_0+\epsilon)\,d\lambda \cdot (\bx-\bx_0) \\
&<&-\frac{c_1^2m_1}{2}\epsilon+c_2m_1 \delta \\
&\le &-\frac{c_1^2m_1}{4}\epsilon  \ < \ 0\\
\text{ and }G(\bx,\theta_0-\epsilon)&>&\frac{c_1^2m_1}{4}\epsilon  \ >\ 0. 
\end{eqnarray*}
Since $G(\bx,\theta)$ is strictly decreasing with respect to $\theta$ in $B_\epsilon(\theta_0)$ by \eqref{Ftheta}  , the intermediate value theorem implies that there is a unique $\theta^*\in (\theta_0-\epsilon,\theta_0+\epsilon)$ such that $G(\bx,\theta^*)=0$, which defines a function $p_{\bx_0}(\bx)=\theta^*$. The statement for $\epsilon/2$ and $\delta/2$ follows similarly.
The smoothness of $p_{\bx_0}$ follows by the classical implicit function theorem, since $G\in C^{\sigma-1}$.
\end{proof}

Now we want to show the uniqueness of the function in a suitable neighborhood $\tU$ of $\Omega$. Denote the minimal period of the periodic orbit by $T$; we can assume that $\epsilon<T$.
Define
$$c:=\min_{\bp\in \Omega}\min_{\theta\in [-T/2,T/2]\setminus (-\epsilon/2,\epsilon/2)}\|S_\theta\bp-\bp\|>0.$$
We can conclude that if $\|S_\theta\bp-\bp\|\le c/2$ with $\bp\in \Omega$ and $|\theta|\le  T/2$, then $|\theta|<\epsilon/2$.

Let $\delta'=\min(\delta/2,c/4)$. Since $\Omega$ is compact and $\Omega\subset \bigcup_{\bx_0\in \Omega}B_{\delta'}(\bx_0)$, there is a finite number of $\bx_i=S_{\theta_i}\bq\in \Omega$, $i=1\ldots,N$, with 
\begin{eqnarray}
	\Omega&\subset& \bigcup_{i=1}^N B_{\delta'}(\bx_i)=:\tU,\label{cover}
	\end{eqnarray}
such that $\tU$ is an open neighborhood of $\Omega$.
We want to show that the $p_{\bx_i}=p_i$ define a unique function $p\colon \tU\to  S^1_T$, where $S^1_T$ are the reals modulo $T$ such that $p=p_i$ on $B_{\delta'}(\bx_i)$. We need to show that if $\bx\in B_{\delta'}(\bx_i)\cap B_{\delta'}(\bx_j)$, then $p_i(\bx)=p_j(\bx)$.

Let  $\bx\in B_{\delta'}(\bx_i)\cap B_{\delta'}(\bx_j)$ and, without loss of generality,  $|\theta_j-\theta_i|\le T/2$ since the $\theta_i$ and $\theta_j$ are uniquely defined only modulo $T$. Then
\begin{eqnarray*}
\|\bx_i-S_{\theta_j-\theta_i}\bx_i\|&=&
\|\bx_i-\bx_j\|\\&\le &\|\bx_i-\bx\|+\|\bx-\bx_j\|\\
&<&2\delta'\ \le\  \min(\delta,c/2).
\end{eqnarray*}
Hence, $|\theta_j-\theta_i|<\epsilon/2$.

Since $\bx\in B_{\delta/2}(\bx_i)\cap B_{\delta/2}(\bx_j)$, we have $p_i(\bx)\in B_{\epsilon/2}(\theta_i)$ and  $p_j(\bx)\in B_{\epsilon/2}(\theta_j)$ by Lemma \ref{implicit}. 
 Then
\begin{eqnarray*}
|p_i(\bx)-\theta_j| 
&\le&|p_i(\bx)-\theta_i| +|\theta_i-\theta_j| \
< \ \epsilon
\end{eqnarray*}
and similarly $p_j(\bx)\in B_\epsilon(\theta_i)$. 
Moreover, $\bx\in B_{\delta}(\bx_j)\cap B_{\delta}(\bx_j)$.   Lemma \ref{implicit} implies that $\theta=p_i(\bx)$ if and only if $G(\bx,\theta)=0$ if and only if $\theta=p_j(\bx)$, which shows $p_i(\bx)=p_j(\bx)$.

Since $\Omega$ is stable, we can choose $\Omega\subset U^\circ\subset U\subset \tU$ such that $U$ is compact and positively invariant. For $\bx\in U$ define $\pi(\bx)=S_{p(\bx)}\bq$. Since $p$ is defined by $p_{\bx_i}$, we have by Lemma \ref{implicit} $0=G(\bx,p(\bx))=(\bx-\pi(\bx))^TM_0(\pi(\bx))\bff(\pi(\bx))$.

If $\bx=S_\theta\bq \in \Omega$, then there is a $\bx_i=S_{\theta_i}\bq\in \Omega$ by \eqref{cover} such that $\bx\in B_{\delta'}(\bx_i)$ and thus, as above, $|\theta-\theta_i|<\epsilon/2$. Hence, by Lemma \ref{implicit}, as $\bx\in B_\delta(\bx_i)$ and $\theta\in B_\epsilon(\theta_i)$, $p_i(\bx)=\theta$ and thus $\pi(\bx)=\bx$, as this satisfies $0=G(\bx,\theta)$. If $\bx\not\in \Omega$, then, since $\pi(\bx)\in \Omega$, $\bx\not=\pi(\bx)$. This shows the lemma.
\end{proof}
%
%%---------------------------------
%It is clear that $\pi({\bf x})={\bf x}$ for all ${\bf x}\in \Omega$.
%We have that 
%\begin{eqnarray*}
%\frac{d}{d\theta}( {\bf x}-S_\theta \bq)^TM_0(S_\theta \bq)\bff(S_\theta \bq)
%&=&-\bff(S_\theta \bq)^TM_0(S_\theta\bq)\bff(S_\theta \bq)\\
%&&+
%( {\bf x}-S_\theta \bq)^TM_0'(S_\theta \bq)\bff(S_\theta \bq)\\
%&&
%+( {\bf x}-S_\theta \bq)^TM_0(S_\theta \bq)D\bff(S_\theta
%\bq)\bff(S_\theta \bq)\,.
%\end{eqnarray*} Since the right-hand side term is $\le -\min_{\theta \in
%[0,T]}\bff(S_\theta \bq)^TM_0(S_\theta\bq)\bff(S_\theta \bq)<0$ for $\bx\in\Omega$, there is a neighborhood
%$U$ of $\Omega$, where the right-hand side term is negative. Hence, the
%Implicit Function Theorem yields the existence and uniqueness of
%$\pi\in C^\sigma(U,\Omega)$, where $U$ is an open neighborhood of $\Omega$.
%Since $\Omega$ is stable, we can choose $U$ positively invariant
%without loss of generality.
%

%\vspace{0.3cm}
%\newpage
\noindent \underline{\bf III. Synchronization}

\noindent In this step we synchronize the time between the solution $S_t\bx$ and the solution on the periodic orbit $S_\theta \pi(\bx)$ such that \eqref{first} holds. This will enable us to later define a distance between $S_t\bx$ and $\Omega$ in Step IV.
\begin{definition}
		For $\bx\in U$ we can define $\theta_\bx\in C^{\sigma-1}(\mathbb R^+_0,\mathbb R)$ by $\theta_\bx(0)=0$ and 
	\begin{eqnarray}
		S_{\theta_\bx(t)}\pi(\bx)&=&\pi(S_t \bx)\label{first}
		\end{eqnarray}
	for all $t\ge 0$.
	
	We have \begin{eqnarray}
	\dot{\theta}_\bx (t)
		&=&(\bff(S_t{\bf x})^TM_0(S_{\theta_\bx(t)} \pi(\bx))\bff(S_{\theta_\bx(t)} \pi(\bx)))\nonumber\\
		&&\bigg(\bff(S_{\theta_\bx(t)} \pi(\bx))^TM_0(S_{\theta_\bx(t)} \pi(\bx))\bff(S_{\theta_\bx(t)} \pi(\bx))\nonumber\\
		&& -
		( S_t{\bf x}-S_{\theta_\bx(t)} \pi(\bx))^T
		\big[M_0'(S_{\theta_\bx(t)} \pi(\bx)) \bff(S_{\theta_\bx(t)} \pi(\bx))\nonumber\\
	&&\hspace{0.5cm}	+M_0(S_{\theta_\bx(t)} \pi(\bx))D\bff(S_{\theta_\bx(t)} \pi(\bx)) \bff(S_{\theta_\bx(t)} \pi(\bx))\big]\bigg)^{-1}.\label{eqbruch}
	\end{eqnarray}	
	The denominator of \eqref{eqbruch} is strictly positive for all $t\ge 0$ and $\bx \in U$.
\end{definition}
\begin{proof}
Denote $\pi(\bx)=:\bp\in \Omega$. Observe, that both sides of \eqref{first} equal for $t=0$.
 For any $t\ge 0$, $S_t\bx\in U$ and  $\pi(S_t \bx)$ denotes a point on the periodic orbit, so we can write it as
$\pi(S_t \bx)=S_{\theta_\bx(t)}\bp$. Note that $\theta_\bx(t)$ is only uniquely defined modulo $T$, however, it is uniquely defined by the requirement that $\theta_\bx$  is a continuous function.

By \eqref{proj}, we have
$$(S_t \bx-S_{\theta_\bx(t)}\bp)^TM_0(S_{\theta_\bx(t)} \bp)\bff(S_{\theta_\bx(t)}\bp)=0.$$
Hence, $\theta_\bx(t)$ is implicitly defined by
 \begin{eqnarray}
 	Q(t,\theta)&=&(S_t{\bf x}-S_\theta \bp)^TM_0(S_\theta \bp)\bff(S_\theta \bp)=0.
 	\label{theta}
 \end{eqnarray}
 Note that $\theta_\bx\in C^{\sigma-1}(\mathbb R^+_0,\mathbb R)$ by the Implicit
 Function Theorem which implies 
  \begin{eqnarray*}
 	\frac{ d\theta_\bx}{dt}
 	&=&-\frac{\partial_t Q(t,\theta)}{\partial_\theta Q(t,\theta)}\bigg|_{\theta=\theta_\bx(t)}
	\nonumber\\
 	&=&(\bff(S_t{\bf x})^TM_0(S_{\theta_\bx(t)} \bp)\bff(S_{\theta_\bx(t)} \bp))\nonumber\\
 	&&\bigg(\bff(S_{\theta_\bx(t)} \bp)^TM_0(S_{\theta_\bx(t)} \bp)\bff(S_{\theta_\bx(t)} \bp)\nonumber\\
 	&&\hspace{0.3cm}-
 	( S_t{\bf x}-S_{\theta_\bx(t)} \bp)^TM_0'(S_{\theta_\bx(t)} \bp) \bff(S_{\theta_\bx(t)} \bp)\nonumber\\
 	&&\hspace{0.3cm}-( S_t{\bf x}-S_{\theta_\bx(t)} \bp)^TM_0(S_{\theta_\bx(t)} \bp)D\bff(S_{\theta_\bx(t)} \bp) \bff(S_{\theta_\bx(t)} \bp)\bigg)^{-1}.
 	 \end{eqnarray*}
	 
	With the notations of the proof of Lemma \ref{imp}, for $S_t\bx\in U$ there is a point $\bx_i=S_{\theta_i} \bq \in \Omega$ such that $S_t\bx\in B_{\delta'}(\bx_i)$. We have $S_t\bx\in B_\delta(\bx_i)$ and, modulo $T$, we have $p_i(S_t\bx)=\theta_\bx(t)\in B_\epsilon(\theta_i)$. Hence, the denominator is $>\frac{c_1^2m_1}{2}$ by \eqref{Ftheta}.
 \end{proof}
	
\begin{lemma}\label{change}
For $\bx\in U$ we have
%$$\pi(S_\tau \bx)=S_{\theta_\bx(\tau)}\pi(\bx)$$
%for all $\tau\ge 0$ and
$$S_{\theta_{S_\tau \bx}(t)}\pi(S_\tau \bx)=S_{\theta_\bx(t+\tau)}\pi(\bx)$$
for all $t,\tau\ge 0$.
\end{lemma}
\begin{proof}
We apply \eqref{first} to the point $S_\tau \bx$ and the time $t$, obtaining
$$S_{\theta_{S_\tau \bx}(t)}\pi(S_\tau \bx)=\pi(S_t S_\tau \bx).$$
Now we apply \eqref{first} to the point $ \bx$ and the time $t+\tau$, obtaining
$$S_{\theta_\bx(t+\tau)}\pi(\bx)=\pi(S_{t+\tau} \bx).$$
As both right-hand sides are equal by the semi-flow property, this proves the % second
statement.
\end{proof}

\vspace{0.3cm}
\noindent \underline{\bf IV. Distance to the periodic orbit}

\noindent
In the following lemma we define a distance of points in $U$ to the periodic orbit, and we show that it decreases exponentially.
%
%Note that the following lemma shows both the desired result for the other paper, using an equivalent norm, and that  \eqref{eqbruch} is not zero if $U$ is small enough.
%SHOW BOTH???
%

\begin{lemma}\label{def_d}
Let $\epsilon<\min(1,\nu/2)$ and $\sigma\ge 2$. Then there is a positively invariant, compact neighborhood $U$ of the periodic orbit $\Omega$ such that
	the function $d\in C^{\sigma-1}(U,\mathbb R^+_0)$, defined by 	
%	For any $\epsilon>0$ there is a positively invariant, compact neighborhood $U$ of the periodic orbit $\Omega$ such that 
$$d(\bx)=(\bx-\pi(\bx))^TM_0(\pi(\bx))(\bx-\pi(\bx))$$
 satisfies $d(\bx)=0$ if and only if $\bx\in \Omega$. Moreover, $d'(\bx)<0$ for all $\bx\in U\setminus \Omega$
and 
\begin{eqnarray*}
d(S_t\bx)&\le& e^{2(-\nu+2\epsilon) t}d(\bx)\text{    for all $\bx\in U$ and all $t\ge 0$,}\\
1-\epsilon&\le& \dot{\theta}_\bx(t)\ \le  \ 1+\epsilon\text{  for all $t\ge 0$}.
\end{eqnarray*}
\end{lemma}

\begin{proof}
Note that $d$ is $C^{\sigma-1}$ as all of its terms are. As $M_0(\bx)$ is positive definite, $d(\bx)=0$ if and only if $\bx=\pi(\bx)$, i.e. $\bx\in \Omega$ by Lemma \ref{imp}.
Define
\begin{eqnarray}
c^*&:=&\frac{ \epsilon}{2 p_1\, p_2 \|S^{-1}\| \,  \|S\| }>0,
\label{defc}\\
c_4&=&2c_2\frac{2c_3m_2+m_3+c^*m_2}{c_1^2m_1},\label{c4def}\\
v^*&:=&\frac{\epsilon}{2 p_1\, p_2 \|S^{-1}\| \,  \|S\| \, c_4(c_3+\|B\|)}\,.\label{defv}
\end{eqnarray}
where the constants where defined in Step II, proof of Lemma \ref{imp}.

For $\by\in U$ we use the Taylor expansion around $ \pi(\by)\in \Omega$. Hence, there is a function $\bpsi(\by)$ satisfying
\begin{eqnarray}
\bff(\by)&=&\bff(\pi(\by))+D\bff(\pi(\by))(\by-\pi(\by))+\bpsi(\by)\label{Taylor}
\end{eqnarray}
with $\|\bpsi(\by)\|\le c^* \|\by-\pi(\by)\|$ for all $\by \in U$, noting that $\Omega$ is compact,
 where we choose $U$ still to be a positively invariant, compact neighborhood of $\Omega$, possibly smaller than before and such that also have 
\begin{eqnarray}
\|\by-\pi(\by)\|&\le&\delta'=\min\left(v^*, \frac{c_1^2m_1}{2c_2[m_3+ m_2c_3] },\frac{\epsilon}{c_4},1\right)\text{ for all $\by\in U$.}\label{vest}
\end{eqnarray}

Recall that, due to the definition of $M_0$ and \eqref{first} we have
\begin{eqnarray*}
d(\bx)&=&(\bx-\pi(\bx))^T(P^{-1}(\pi(\bx)))^*(S^{-1})^*S^{-1}P^{-1}(\pi(\bx))(\bx-\pi(\bx))\\
d(S_t\bx)&=&(S_t\bx-S_{\theta_\bx(t)}\pi(\bx))^T(P^{-1}(S_{\theta_\bx(t)}\pi(\bx)))^*(S^{-1})^*\\
&&S^{-1}P^{-1}(S_{\theta_\bx(t)}\pi(\bx))(S_t\bx-S_{\theta_\bx(t)}\pi(\bx)).
\end{eqnarray*}

Now let us calculate the orbital derivative, denoting 
$\theta(t):=\theta_\bx(t)$.
%$\bp:=\pi(\bx)$, $\bv:=\bx-\bp$
\begin{eqnarray*}
d'(S_t\bx)&=&\bigg[\frac{d}{dt}\left(P^{-1}(S_{\theta(t)}\pi(\bx))\right)(S_t\bx-S_{\theta(t)}\pi(\bx))\\
&&+P^{-1}(S_{\theta(t)}\pi(\bx))[\bff(S_t\bx)-\dot{\theta}(t)\bff(S_{\theta(t)}\pi(\bx))]\bigg]^*\\
&&(S^{-1})^*S^{-1}P^{-1}(S_{\theta(t)}\pi(\bx))(S_t\bx-S_{\theta(t)}\pi(\bx))\\
&&+
(S_t\bx-S_{\theta(t)}\pi(\bx))^T(P^{-1}(S_{\theta(t)}\pi(\bx)))^*(S^{-1})^*S^{-1}\\
&&
\bigg[\frac{d}{dt}\left(P^{-1}(S_{\theta(t)}\pi(\bx))\right)(S_t\bx-S_{\theta(t)}\pi(\bx))\\
&&+P^{-1}(S_{\theta(t)}\pi(\bx))[\bff(S_t\bx)-\dot{\theta}(t)\bff(S_{\theta(t)}\pi(\bx))]\bigg].
\end{eqnarray*}
We denote $\bp:=\pi(\bx)$ and  $\bv(t):=S_t\bx-S_{\theta(t)}\pi(\bx)=S_t\bx-\pi(S_t\bx)$ by  \eqref{first}.
Hence, using \eqref{vest} for $\by=S_t\bx\in U$ since $\bx\in U$, which is positively invariant, we have 
\begin{eqnarray}
\|\bv(t)\|&\le&\delta'=\min\left(v^*, \frac{c_1^2m_1}{2c_2[m_3+ m_2c_3] },\frac{\epsilon}{c_4},1\right)\label{vest2}
\end{eqnarray}
for all $t\ge 0$. 
 We have 
 $\frac{d}{dt}\left(P^{-1}(S_{\theta(t)}\pi(\bx))\right)=\dot{\theta}(t)(-P^{-1}(S_{\theta(t)}\bp)D\bff(S_{\theta(t)}\bp)+B P^{-1}(S_{\theta(t)}\bp))$ by \eqref{alto}.
 Thus, 
\begin{eqnarray}
d'(S_t\bx)&=&
\bigg[\dot{\theta}(t)(-P^{-1}(S_{\theta(t)}\bp)D\bff(S_{\theta(t)}\bp)+B P^{-1}(S_{\theta(t)}\bp))\bv(t)\nonumber\\
&&+P^{-1}(S_{\theta(t)}\bp)[\bff(S_t\bx)-\dot{\theta}(t)\bff(S_{\theta(t)}\bp)]\bigg]^*\nonumber\\
&&(S^{-1})^*S^{-1}P^{-1}(S_{\theta(t)}\bp)\bv(t)\nonumber\\
&&+\bv(t)^T(P^{-1}(S_{\theta(t)}\bp))^*(S^{-1})^*S^{-1}\nonumber\\
&&
\bigg[\dot{\theta}(t)(-P^{-1}(S_{\theta(t)}\bp)D\bff(S_{\theta(t)}\bp)+B\nonumber P^{-1}(S_{\theta(t)}\bp))\bv(t)\nonumber\\
&&+P^{-1}(S_{\theta(t)}\bp)[\bff(S_t\bx)-\dot{\theta}(t)\bff(S_{\theta(t)}\bp)]\bigg].\label{eq33}
\end{eqnarray}

Using the Taylor expansion \eqref{Taylor} for $\by=S_t\bx$, we obtain  with $\pi(S_t\bx)=S_{\theta(t)}\bp$,
\begin{eqnarray}
\bff(S_t\bx)&=&\bff(S_{\theta(t)}\bp)+D\bff(S_{\theta(t)}\bp) \bv(t)+\bpsi(S_t\bx)
\label{Taylor2}
\end{eqnarray}
and thus with \eqref{eqbruch}
 \begin{eqnarray*}
 	\lefteqn{
 \dot{\theta}(t)-1}\nonumber\\
&=&\bigg( \bff(S_t\bx)^TM_0(S_{\theta(t)}\bp)\bff(S_{\theta(t)}\bp)-\bff(S_{\theta(t)}\bp)^TM_0(S_{\theta(t)}\bp)\bff(S_{\theta(t)}\bp)\\
&&\hspace{0.5cm}+\bv(t)^TM_0'(S_{\theta(t)}\bp)\bff(S_{\theta(t)}\bp)+\bv(t)^TM_0(S_{\theta(t)}\bp)D\bff(S_{\theta(t)}\bp)\bff(S_{\theta(t)}\bp)\bigg)\\
&&\bigg(\bff(S_{\theta(t)}\bp)^TM_0(S_{\theta(t)}\bp)\bff(S_{\theta(t)}\bp)-\bv(t)^TM_0'(S_{\theta(t)}\bp)\bff(S_{\theta(t)}\bp)\\
&&\hspace{0.5cm}-\bv(t)^TM_0(S_{\theta(t)}\bp)D\bff(S_{\theta(t)}\bp)\bff(S_{\theta(t)}\bp)\bigg)^{-1}\\
 &=&\bigg( \bv(t)^TD\bff(S_{\theta(t)}\bp)^TM_0(S_{\theta(t)}\bp)\bff(S_{\theta(t)}\bp)+\bpsi(S_t\bx)^TM_0(S_{\theta(t)}\bp)\bff(S_{\theta(t)}\bp)\\
  &&+\hspace{0.5cm}  \bv(t)^TM_0'(S_{\theta(t)}\bp)\bff(S_{\theta(t)}\bp)+\bv(t)^TM_0(S_{\theta(t)}\bp)D\bff(S_{\theta(t)}\bp)\bff(S_{\theta(t)}\bp)\bigg)\\
&&\bigg(\bff(S_{\theta(t)}\bp)^TM_0(S_{\theta(t)}\bp)\bff(S_{\theta(t)}\bp)-\bv(t)^TM_0'(S_{\theta(t)}\bp)\bff(S_{\theta(t)}\bp)\\
&&\hspace{0.5cm}-\bv(t)^TM_0(S_{\theta(t)}\bp)D\bff(S_{\theta(t)}\bp)\bff(S_{\theta(t)}\bp)\bigg)^{-1}
 \end{eqnarray*}
 which shows, using \eqref{vest2} and \eqref{c4def},
 \begin{eqnarray}
  | \dot{\theta}(t)-1|&\le&\frac{\|\bv(t)\| c_2[2c_3 m_2 +m_3]+
  \|\bpsi(S_t\bx)\| m_2c_2  }{c_1^2m_1-\|\bv(t)\|c_2[m_3+ m_2c_3] }\nonumber\\
   &\le&2c_2\frac{ 2c_3 m_2 +m_3+
  c^* m_2}{c_1^2m_1 }\|\bv(t)\|=c_4\|\bv(t)\|\le \epsilon\,.\label{est}
  \end{eqnarray}
 In particular, we have
  $1-\epsilon\le  \dot{\theta}(t)\le 1+\epsilon$, which shows the existence of $\theta(t)$ for all $t\ge 0$, $\dot{\theta}(t)>0$ for all $t\ge 0$, that $\theta(t)$ is a bijective function from $[0,\infty)$ to $[0,\infty)$ and $\lim_{t\to\infty}\theta(t)=\infty$.

  Hence, we have from \eqref{eq33} and \eqref{Taylor2}
    \begin{eqnarray*}
d'(S_t\bx)&=&\big[(1-\dot{\theta}(t))P^{-1}(S_{\theta(t)}\bp)D\bff(S_{\theta(t)}\bp)\bv(t)
+B P^{-1}(S_{\theta(t)}\bp)\bv(t)\\
&&\hspace{0.5cm}-(1-\dot{\theta}(t))B P^{-1}(S_{\theta(t)}\bp)\bv(t)+(1-\dot{\theta}(t))
P^{-1}(S_{\theta(t)}\bp)\bff(S_{\theta(t)}\bp)
\\
&&\hspace{0.5cm}+P^{-1}(S_{\theta(t)}\bp)\bpsi(S_t\bx)\big]^*
(S^{-1})^*S^{-1}P^{-1}(S_{\theta(t)}\bp)\bv(t)\\
&&+\bv(t)^T(P^{-1}(S_{\theta(t)}\bp))^*(S^{-1})^*S^{-1}\big[(1-\dot{\theta}(t))P^{-1}(S_{\theta(t)}\bp)D\bff(S_{\theta(t)}\bp)\bv(t)\\
&&\hspace{0.5cm}+B P^{-1}(S_{\theta(t)}\bp)\bv(t)-(1-\dot{\theta}(t))B P^{-1}(S_{\theta(t)}\bp)\bv(t)\\
&&\hspace{0.5cm}
+(1-\dot{\theta}(t))
P^{-1}(S_{\theta(t)}\bp)\bff(S_{\theta(t)}\bp)+P^{-1}(S_{\theta(t)}\bp)\bpsi(S_t\bx)\big]\\
&\le&2
\|S^{-1}P^{-1}(S_{\theta(t)}\bp)\bv(t)\|\, \|S^{-1}\| \,  \|P^{-1}(S_{\theta(t)}\bp)\|\,\\
&&\hspace{0.5cm}
\left[|1-\dot{\theta}(t)|(\|D\bff(S_{\theta(t)}\bp)\| +\|B\|)\|\bv(t)\|+\|\bpsi(S_t\bx)\|\right]\\
&&+\bv(t)^*(P^{-1}(S_{\theta(t)}\bp))^*\left[(S^{-1})^*S^{-1}B +B^*(S^{-1})^*S^{-1}
\right]P^{-1}(S_{\theta(t)}\bp)\bv(t)
\end{eqnarray*}  
using $$0=\bff(S_{\theta(t)}\bp)^*M_0(S_{\theta(t)}\bp)\bv(t)=\bff(S_{\theta(t)}\bp)^*(P^{-1}(S_{\theta(t)}\bp))^*(S^{-1})^*S^{-1}P^{-1}(S_{\theta(t)}\bp)\bv(t)$$ by \eqref{theta}.

Setting $\bw(t)=S^{-1}P(S_{\theta(t)}\bp)^{-1}\bv(t)$, we obtain, using \eqref{est}
and \eqref{vest2}
\begin{eqnarray*}
d'(S_t\bx)
&\le&2\, p_2
\|\bw(t)\|\, \|S^{-1}\|  \|\bv(t)\|
\left[c_4  (c_3+\|B\|) \|\bv(t)\|+c^*\right]\\
&&+\bw(t)^* \left[S^{-1}B S+S^*B^*(S^{-1})^*\right]\bw(t)\\
&\le&2  p_1 p_2 \|S\| \, 
 \|S^{-1}\|  \left[c_4 (c_3+\|B\|)\|\bv(t)\|+c^*\right]\,\|\bw(t)\|^2\\
 &&+\bw(t)^* \left[A+A^*\right]\bw(t)\\
&\le&2\epsilon\,\|\bw(t)\|^2+\bw(t)^* \left[A+A^*\right]\bw(t)
\end{eqnarray*}
by \eqref{defv} and \eqref{defc}.
 Noting that
$$w_1(t)=\be_1^*\bw(t)=\bff(S_{\theta(t)}\bp)^*(P^{-1}(S_{\theta(t)}\bp))^*(S^{-1})^*S^{-1}P^{-1}(S_{\theta(t)}\bp)\bv(t)=0$$
we have   with Corollary \ref{coro}
$$\bw(t)^* \left[A+A^*\right]\bw(t) \le   2 (-\nu+\epsilon)
\|\bw(t)\|^2.$$

Altogether, we have
\begin{eqnarray*}
d'(S_t\bx)
&\le&
\left[  2\epsilon-2\nu+2\epsilon\right]\|\bw(t)\|^2\\
&=&2(-\nu+2\epsilon) d(S_t\bx),
\end{eqnarray*}which shows  $d(S_t\bx)\le e^{2(-\nu+2\epsilon)  t}d(\bx)$ for all $\bx\in U$ and $t\ge 0$.
%It is clear that $\Gamma$ is a non-characteristic hyperplane if $\delta<\min_{\bx\in \partial U}d(\bx)$.
\end{proof}

Let us summarize the results obtained so far in the following corollary.

\begin{corollary}\label{help_other}
Let $\Omega$ be an exponentially stable periodic orbit of $\dot{\bx}=\bff(\bx)$ with $\bff\in C^\sigma(\mathbb R^n,\mathbb R^n)$ and $\sigma\ge 2$, such that
$-\nu<0$ is the maximal real part of all non-trivial Floquet exponents.

	For   $\epsilon_0\in (0,\min(\nu,1))>0$  there is a compact, positively invariant neighborhood $U$ of $\Omega$ with $\Omega\subset U^\circ$ and $U\subset A(\Omega)$, and a map $\pi\in C^{\sigma-1}( U, \Omega)$ with $\pi(\bx)=\bx$ if and only if $\bx\in\Omega$.
	
	Furthermore, for a fixed $\bx\in U$, there is a bijective $C^{\sigma-1}$ map
	$\theta_\bx\colon [0,\infty)\to[0,\infty)$ with inverse
		$t_\bx=\theta_\bx^{-1}\in C^{\sigma-1}( [0,\infty),[0,\infty))$
	such that $\theta_\bx(0)=0$ and
	$$\pi(S_t\bx)=S_{\theta_\bx(t)}\pi(\bx)$$
	for all $t\in [0,\infty)$.
	We have $\dot{\theta}_\bx(t)\in \left[1-\epsilon_0,1+\epsilon_0\right]$ for all $t\ge 0$ and 
	 $\dot{t}_\bx(\theta)\in \left[1-\epsilon_0,1+\epsilon_0\right]$ for all $\theta\ge 0$. 
	 
	 Finally, there is a constant $C>0$ such that 
		\begin{eqnarray}|\dot{t}_\bx(\theta)-1|&\le& Ce^{(-\nu+\epsilon_0) \theta}\label{res1}\\
	\|S_{t_\bx(\theta)}\bx-S_\theta \pi(\bx)\|&\le& Ce^{(-\nu+\epsilon_0)  \theta}\|\bx-\pi(\bx)\|\label{res2}
	\end{eqnarray}
	for all $\theta\ge 0$ and all $\bx\in U$.
	\end{corollary}
\begin{proof}
Setting $\epsilon:=\frac{\epsilon_0}{2(1+\nu)}\le\min\left( \frac{\epsilon_0}{2},\frac{1}{2}\right)\le  \min\left( \frac{\nu}{2},1\right)$, all results  follow directly from Lemma \ref{def_d}   by using the inverse $t(\theta)$ of $\theta(t)$. Indeed,  we have\begin{eqnarray*}
|\dot{t}_\bx(\theta)-1|&=&\left|\frac{1-\dot{\theta}_\bx(t(\theta))}{\dot{\theta}_\bx(t(\theta))}\right|\\
&\le&\frac{\epsilon}{1-\epsilon}\\
&\le&2\epsilon\le\epsilon_0\,.
\end{eqnarray*}
Furthermore, we have by  \eqref{est} and noting that
$ m_1 \|S_{t_\bx(\theta)}\bx-S_\theta \pi(\bx)\|^2\le d(S_{t_\bx(\theta)}\bx) \le m_2 \|S_{t_\bx(\theta)}\bx-S_\theta \pi(\bx)\|^2$  
\begin{eqnarray*}
|\dot{t}_\bx(\theta)-1|&\le&
\left|\frac{1-\dot{\theta}_\bx(t(\theta))}{1/2}\right|\\
&\le&2c_4\|\bv(t(\theta))\|\\
&\le&\frac{2c_4}{\sqrt{m_1}}\sqrt{d(S_{t(\theta)}\bx)}\\
&\le&Ce^{(-\nu+2\epsilon)t(\theta)}\sqrt{d(\bx)}\\
&\le&Ce^{(-\nu+2\epsilon)(1-2\epsilon)\theta}\\
&\le&Ce^{(-\nu+2\epsilon(1+\nu)-4\epsilon^2)\theta}\\
&\le&Ce^{(-\nu+\epsilon_0)\theta},
\end{eqnarray*}
using $t(\theta)=\int_0^\theta \dot{t}(\tau)\,d\tau\ge \theta(1-2\epsilon)$ and that $d(\bx)$ is bounded in $U$.
Similarly, we can prove \eqref{res2} from Lemma \ref{def_d}.
\end{proof}

\vspace{0.3cm}
\noindent \underline{\bf V. Definition of $M_1$ and $M$ in $A(\Omega)$}

\noindent
For all $\bx\in U$ we have defined the distance
$$d(\bx)=(\bx-\pi(\bx))^TM_0(\pi(\bx))(\bx-\pi(\bx))$$
in Lemma \ref{def_d}
which is $C^{\sigma-1}$. Let $\iota>0$ be so small that
the set $\Omega_{2\iota}:=\{\bx\in U\colon d(\bx)\le 2\iota\}$ satisfies $\Omega_{2\iota}\subset U^\circ$.  Define
the $C^\infty$ functions $h_1\colon \Omega_\iota \to [0,1]$, $h_2\colon \Omega_{2\iota} \to [0,1]$ such that $h_1(\bx)=1$ for all $d(\bx)\le \frac{\iota}{3}$ and $h_1(\bx)=0$ for all $d(\bx)\ge \frac{2}{3} \iota$, and $h_2(\bx)=1$ for all $d(\bx)\le \frac{4}{3}\iota$ and $h_2(\bx)=0$ for all $d(\bx)\ge \frac{5}{3} \iota$.  Set
$$M_1(\bx):=\left\{\begin{array}{ll}
I&\mbox{ if }\bx\not \in  \Omega_{2\iota},\\
(1-h_2(\bx))I+h_2(\bx)M_0(\pi(\bx))&\mbox{ if }\bx \in  \Omega_{2\iota}.\end{array}\right.$$
It is clear that $M_1(\bx)$ is positive definite for all $\bx\in \mathbb R^n$, $M_1$ is $C^{\sigma-1}$  and $M_1(\pi(\bx))=M_0(\pi(\bx))$ for all $\bx\in \Omega_{\frac{4}{3}\iota}$.
%
%To define $M_1$ on $\mathbb R^n$, we first define $M_1$ in $U$, by setting
%$M_1(\bx)=M_0(P(\bx))$ and then, using a partition of unity, we can define $M_0$ to be
%$M_0(\bx)=I$ outside and $M_0(\bx)$ as before on $\Omega$, we can check that $M_0$ is positive definite (done somewhere, perhaps with clever partition of unity, distance from $\Omega$????

%\noindent {\bf VI. Definition of $M$}

We will define the Riemannian metric $M$ through $M_1$ and a scalar-valued function $V\colon A(\Omega)\to \mathbb R$, which will be defined later. 
Let us denote $\mu:=\nu-\epsilon>0$. The function $V$ will be continuous and continuously orbitally differentiable and satisfy
\begin{eqnarray}
V'(\bx)&=&-L_{M_1}(\bx)+r(\bx), \text{ where }\label{eqV}\\
r(\bx)&=&\left\{\begin{array}{ll}
-\mu&\mbox{ if }\bx\not \in  \Omega_\iota,\\
-\mu(1-h_1(\bx))+h_1(\bx)L_{M_1}(\pi(\bx))&\mbox{ if }\bx \in  \Omega_\iota.\end{array}\right.
\end{eqnarray}
%Recall that we have defined the function $h\colon \Omega_\iota \to [0,1]$ such that $h(\bx)=1$ for all $d(\bx)\le \iota/3$ and $h(\bx)=0$ for all $d(\bx)\ge \frac{2}{3} \iota$.  
Note that $$r(\bx)\le -\mu$$
for all $\bx\in \mathbb R^n$. Indeed, for $\bx\in\Omega_\iota$ we have   $L_{M_1}(\pi(\bx))=L_{M_0}(\pi(\bx))\le -\mu$ as $\pi(\bx)\in \Omega$, see \eqref{LMv}, and thus
\begin{eqnarray*}
r(\bx)&=& -\mu+\underbrace{h_1(\bx)}_{\ge 0}\underbrace{(\mu+L_{M_1}(\pi(\bx))}_{\le 0}
\ \le\  -\mu\,.
\end{eqnarray*} 

Then we define
$$M(\bx)=e^{2V(\bx)}M_1(\bx).$$
We obtain by Lemma \ref{lem}
$$L_M(\bx)=L_{M_1}(\bx)+V'(\bx)
=L_{M_1}(\bx)-L_{M_1}(\bx)+r(\bx)\le-\mu.$$

This shows the theorem. In the last steps we will define the function $V$ and prove the properties stated above.

\vspace{0.3cm}
\noindent \underline{\bf VI. Definition of $V_{loc}$}

\noindent
We define $V_{loc}({\bf x})$ for $ {\bf x}\in \Omega_\iota$. Note that $\Omega_\iota$ is positively invariant by Lemma \ref{def_d}, so
$S_t \bx\in \Omega_\iota$ for all $t\ge 0$. We define
\begin{eqnarray}
V_{loc}({\bf x})&=&\int_0^\infty [L_{M_1}(S_t {\bf x})-L_{M_1}(S_{\theta_\bx(t)}
\pi({\bf x}))]\,dt.\label{Vdefi}
\end{eqnarray}
We have $V_{loc}(\bx)=0$ for all $\bx\in\Omega$.
We will show that the $V_{loc}$ is well-defined, continuous and orbitally continuously differentiable for all $\bx\in \Omega_\iota$ and that
 (\ref{eqV}) holds for all $\bx\in \overline{\Omega_{\iota/3}}$.

For $\bx\in U$, define 
\begin{eqnarray*}
g_T(\tau,\bx)&=&\int_\tau^{T+\tau} [L_{M_1}(S_t {\bf x})-L_{M_1}(S_{\theta_\bx(t)}
\pi({\bf x}))]\,dt. %\label{Vdefi}
\end{eqnarray*}
  By Lemma \ref{def_d} there is a constant $C>0$ such that, defining $\bp:=\pi(\bx)\in \Omega$,
 \begin{eqnarray}
 \label{expsta}
 \|S_t\bx-S_{\theta_\bx(t)}\bp\|&\le& Ce^{-\mu_0 t}
 \end{eqnarray}
 for all $t\ge 0$ and all $\bx\in U$ with $\mu_0:=\nu-2\epsilon>0$; note that $S_{\theta_\bx(t)}\bp=\pi(S_t\bx)$ by \eqref{first}. 

 Now, we use  Lemma \ref{Lipschitz} and $\sigma\ge 3$, showing that $L_{M_1}$ is Lipschitz-continuous on the compact set $U$; note that $\sigma-1\ge 2$. 
 %There is $T_0>0$ such that $\|S_t\bx-\pi(S_t\bx)\|<\tdelta$ for all $t\ge T_0$ and all $\bx\in U$ by Lemma \ref{def_d}. Hence, choosing the constant below such that the inequality also holds for $t\in [0,T_0]$,
 Hence, 
  \begin{eqnarray*}
 \left|L_{M_1}(S_t {\bf x})-L_{M_1}(S_{\theta_\bx(t)}
\pi({\bf x}))\right|
&\le&L  C_1\left\|S_t {\bf x}-S_{\theta_{\bx}(t)}\bp\right\|\\
&\le&L C_2 e^{-\mu_0 t}
\end{eqnarray*}
by \eqref{expsta},
which is integrable over $[0,\infty)$. Hence, by Lebesgue's dominated convergence theorem, the function $g_T(\tau,\bx)$ converges point-wise for $T\to \infty$ for all $\tau\ge 0$  and $\bx\in U$.

Choose $\theta_0>0$ so small that $S_{-\theta_0}\Omega_{\iota}\subset U$.
We have that
\begin{eqnarray*}
\lefteqn{\frac{\partial }{\partial \tau}g_T(\tau,\bx)}\\&=&
[L_{M_1}(S_{T+\tau} {\bf x})-L_{M_1}(S_{\theta_\bx(T+\tau)}
\pi({\bf x}))]
- \left(L_{M_1}(S_{\tau}\bx)-L_{M_1}(S_{\theta_{\bx}(\tau)}\bp)\right)\\
&=&
[L_{M_1}(S_T(S_\tau {\bf x}))-L_{M_1}(S_{\theta_{S_\tau\bx}(T)}
\pi(S_\tau{\bf x}))]
- \left(L_{M_1}(S_{\tau}\bx)-L_{M_1}(S_{\theta_{\bx}(\tau)}\bp)\right)
\end{eqnarray*}
by Lemma \ref{change}. For $\bx\in\Omega_\iota$, the right-hand side  
converges uniformly in $\tau \in  (-\theta_0,\theta_0)$ 
as $T\to \infty$ to
$- \left(L_{M_1}(S_{\tau}\bx)-L_{M_1}(S_{\theta_{\bx}(\tau)}\bp)\right)$ by the same estimate as above. Hence, we can exchange $\frac{d}{d\tau}$ and $\lim_{T\to\infty}$.
Altogether, we thus have for all $\bx\in\Omega_\iota$, using Lemma \ref{change}
\begin{eqnarray*}
V_{loc}'(\bx)&=&\frac{d}{d\tau}V_{loc}(S_\tau \bx)\bigg|_{\tau=0}\\
&=&\frac{d}{d\tau}\int_0^\infty [L_{M_1}(S_{t+\tau} {\bf x})-L_{M_1}(S_{\theta_{S_\tau\bx}
(t)}
\pi(S_\tau{\bf x}))]\,dt\bigg|_{\tau=0}\\
&=&\frac{d}{d\tau}\lim_{T\to\infty}\int_0^T  [L_{M_1}(S_{t+\tau} {\bf x})-L_{M_1}(S_{\theta_\bx(t+\tau)}
\pi(\bx))]\,dt\bigg|_{\tau=0}\\
&=&\frac{d}{d\tau}\lim_{T\to\infty}\int_\tau^{T+\tau} [L_{M_1}(S_{t} {\bf x})-L_{M_1}(S_{\theta_\bx (t)}
\pi(\bx))]\,dt\bigg|_{\tau=0}\\
&=&\frac{d}{d\tau}\lim_{T\to \infty}g_T(\tau,\bx)\bigg|_{\tau=0}\\
&=&\lim_{T\to \infty}\frac{d}{d\tau}g_T(\tau,\bx)\bigg|_{\tau=0}\\
&=&- L_{M_1}(\bx)+L_{M_1}(\bp)
\end{eqnarray*}
and in particular, that $V_{loc}$ is continuously orbitally differentiable. 
Note that $V_{loc}'(\bx)=-L_{M_1}(\bx)+r(\bx)$ 
for all $\bx\in\overline{\Omega_{\iota/3}}$.

%\vspace{0.3cm}
%\newpage
\noindent \underline{\bf VII. Definition of $V_{glob}$ in $A(\Omega)$}

\noindent
For the global part note that 
$V_{loc}$ is defined and smooth in $\Omega_{\iota}$ and we have
$V_{loc}'(\bx)=-L_{M_1}(\bx)+r(\bx)$ for all $\bx\in \overline{\Omega_{\iota/3}}$.
The global function $V_{glob}\colon A(\Omega)\setminus \Omega\to \mathbb R$ is defined as the solution of the non-characteristic Cauchy problem
\begin{eqnarray}
\begin{array}{lcl}
  \nabla V_{glob}({\bf x})\cdot \bff({\bf x})&=&-L_{M_1}(\bx)+r(\bx)
\mbox{  for }{\bf x}\in A(\Omega)\setminus
\Omega\\
V_{glob}({\bf x})&=&V({\bf x}) \mbox{ for }
{\bf x}\in\Gamma,\end{array}\label{cauchy}
\end{eqnarray} where
$\Gamma=\{{\bf x}\in U\mid d({\bf x})=\iota/3\}$.
%Note that by this lemma, $\Gamma$ is a non-characteristic manifold.

 In particular, we can construct the solution by first defining the function $\tau\in C^\sigma(A(\Omega)\setminus \Omega,\mathbb R)$ implicitly by
$$d(S_{\tau}\bx)=\iota/3.$$
Since $\bx\in A(\Omega)\setminus \Omega$,  there exists a $\tau$ satisfying the equation, and since $d'(\bx)<0$ for all $\bx\in \Gamma$, $\tau(\bx)$ is unique. The function $\tau$ is $C^{\sigma-1}$, since $d$ and $S_\tau$ are. We have
$\tau'(\bx)=%-\frac{\nabla d(S_\tau \bx)\cdot \bff(S_\tau\bx)}{d'(S_\tau \bx)}=
-1$. Then the function 
$$V_{glob}(\bx)=\int_0^{\tau(\bx)} q(S_t\bx)\,dt+V_{loc}(S_{\tau(\bx)}(\bx))$$
with $q(\bx):=L_{M_1}(\bx)-r(\bx)$ is continuous and orbitally continuously differentiable and satisfies \eqref{cauchy}, noting that $S_{\tau(\bx)}(\bx)=S_{\tau(S_\theta\bx)}(S_\theta\bx)$ for all $\theta\ge 0$. Indeed, for $\bx\in \Gamma$ we have $ V_{glob}(\bx)=V_{loc}(\bx)$ and we have
\begin{eqnarray*}
V_{glob}'(\bx)&=&\frac{d}{d\theta}\left(\int_0^{\tau(S_\theta \bx)} q(S_{t+\theta}\bx)\,dt+V(S_{\tau(S_\theta \bx)}(S_\theta \bx))\right)\bigg|_{\theta=0}\\
&=&\frac{d}{d\theta}\left(\int_\theta^{\tau(S_\theta \bx)+\theta} q(S_{t}\bx)\,dt+V(S_{\tau( \bx)}( \bx))\right)\bigg|_{\theta=0}\\
&=&\left( q(S_{\tau(S_\theta \bx)+\theta}\bx)(\tau'( \bx)+1)
-q(S_\theta \bx)\right)\bigg|_{\theta=0}\\
&=&-q(\bx)
	\end{eqnarray*}
	since $\tau'(\bx)=-1$.
	
	Note that we have $V_{glob}(\bx)=V_{loc}(\bx)$ for $\bx\in \overline{\Omega_{\iota/3}}\setminus \Omega$, and hence $V_{glob}$ can be extended to a continuous and orbitally continuously differentiable function $V$ on  $A(\Omega)$ satisfying \eqref{eqV} by setting $V_{glob}(\bx):=V_{loc}(\bx)=0$ for all $\bx \in \Omega$.
This proves the theorem.
\end{proof}

\section*{Conclusions}

In this paper we have proven a converse theorem, showing the existence of a contraction metric for an exponentially stable periodic orbit. The metric is defined in its basin of attraction and the bound  on the function $L_M$ is arbitrarily close to the true  exponential rate of attraction.

\begin{appendix}
	
	\section{Local Lipschitz-continuity of $L_M$}
	
	In the appendix we prove that the function $L_M$ is locally Lipschitz continuous.
	 
\begin{lemma}\label{Lipschitz} Let $\bff\in C^2(\mathbb R^n,\mathbb R^n)$ and $M\in C^2(\mathbb R^n,\mathbb S^n)$ such that $M(\bx)$ is positive definite for all $\bx\in \mathbb R^n$.

 Then $L_M$ is locally Lipschitz continuous on $D=\{\bx\in \mathbb R^n\mid \bff(\bx)\not=\bnull\}$.
 %, i.e. for all $\bx\in D$
%there exists a constant $L>0$ and $\tdelta>0$ such that 
% \begin{eqnarray*}
% \left|L_{M}( {\bf y})-L_{M}(\bx)\right|
%&\le&L  \left\| {\bf y}- \bx\right\|
%\end{eqnarray*}
%holds for all $\by\in D$ with $\|\bx-\by\|<\tdelta$.
\end{lemma}
\begin{proof}
For $\by\in D$ we define a projection $P_\by\colon \mathbb R^n\to\mathbb R^n$  onto the $(n-1)$-dimensional space of vectors $\bw\in \mathbb R^n$ with $\bff(\by)^TM(\by)\bw=0$  by 
$$P_\by \bv =\bv -\frac{\bff(\by)^TM(\by)\bv}{\bff(\by)^TM(\by)\bff(\by)}\bff(\by)$$
for all $\by\in D$ and all $\bv\in \mathbb R^n$.
Note that indeed
\begin{eqnarray*}
\bff(\by)^TM(\by)P_\by \bv 
&=&\bff(\by)^TM(\by)\bv -\frac{\bff(\by)^TM(\by)\bv}{\bff(\by)^TM(\by)\bff(\by)}
\bff(\by)^TM(\by)\bff(\by)\\
&=&\bnull.
\end{eqnarray*}

%To show that $L_M$ is locally Lipschitz continuous it is sufficient to show this for all $\by\in B_\epsilon(\bx)$, where $\bx\in D$ is an arbitrary point. Hence,
Fix $\bx\in D$ and choose a basis $\bv_1=\bff(\bx),\bv_2,\ldots,\bv_n$ of $\mathbb R^n$ such that $\bv_i^TM(\bx)\bv_j=0$ for $i\not=j$.
Choose $\epsilon>0$ such that 
\begin{eqnarray}
\bff(\by)^TM(\bx)\bff(\bx)&\not=&0\label{not0}
\end{eqnarray} holds for all $\by\in B_\epsilon(\bx)$; note that for $\by=\bx$ we have $\bff(\bx)^TM(\bx)\bff(\bx)\not=0$. 

For $\by\in B_\epsilon(\bx)$ we define $\bw_1=\bff(\by)$ and $\bw_i=P_\by\bv_i$ for $i=2,\ldots,n$. We show that $(\bw_1,\ldots,\bw_n)$ is a basis of $\mathbb R^n$.

Let us first show that $\bw_i\not=\bnull$ for $i=2,\ldots,n$. Assuming the opposite, we have
\begin{eqnarray}
\bv_i&=&\frac{\bff(\by)M(\by)\bv_i}{\bff(\by)^TM(\by)\bff(\by)}\bff(\by)\label{first_eq}\\
0&=&\frac{\bff(\by)M(\by)\bv_i}{\bff(\by)^TM(\by)\bff(\by)}\nonumber
\end{eqnarray}
 multiplying by $\bff(\bx)^TM(\bx)$ from the left
as $\bff(\bx)^TM(\bx)\bff(\by)\not=0$ by \eqref{not0}. This, however, implies by \eqref{first_eq} that $\bv_i=\bnull$ which is a contradiction.
$\bw_1\not=\bnull$ follows directly from \eqref{not0}.

We express $\bff(\by)=\sum_{j=1}^n\beta_j\bv_j$ and note that multiplying this equation by $\bff(\bx)^TM(\bx)$ from the left gives $$0\not=\bff(\bx)^TM(\bx)\bff(\by)=\beta_1 \bff(\bx)^TM(\bx)\bff(\bx)$$ 
by \eqref{not0}, i.e. in particular $\beta_1\not=0$.

To show that the $\bw_i$ form a basis, we  assume $\sum_{i=1}^n\alpha_i\bw_i=\bnull$. Multiplying this equation by $\bff(\by)^TM(\by)$ from the left gives
$\alpha_1 \bff(\by)^TM(\by)\bff(\by)=0$ by the projection property, hence $\alpha_1=0$.

Hence,
\begin{eqnarray*}
\bnull&=&\sum_{i=2}^n\alpha_i\left[\bv_i-\frac{\bff(\by)^TM(\by)\bv_i}{\bff(\by)^TM(\by)\bff(\by)}\bff(\by)\right]\\
&=&\sum_{i=2}^n\alpha_i\bv_i-
\sum_{i=2}^n\sum_{j=1}^n \frac{\bff(\by)^TM(\by)\bv_i}{\bff(\by)^TM(\by)\bff(\by)}\beta_j\bv_j.
\end{eqnarray*}
Using that $\bv_j$ is a basis, we can conclude that the coefficient in front of $\bv_1$ is zero, namely
$$
\sum_{i=2}^n \frac{\bff(\by)^TM(\by)\bv_i}{\bff(\by)^TM(\by)\bff(\by)}\beta_1=0.$$
Since $\beta_1\not=0$, we have $
\sum_{i=2}^n \frac{\bff(\by)^TM(\by)\bv_i}{\bff(\by)^TM(\by)\bff(\by)}=0$.
Plugging this back in, we obtain $\sum_{i=2}^n\alpha_i\bv_i=\bnull$, which shows $\alpha_2=\ldots=\alpha_n=0$ as the $\bv_i$ are linearly independent.

Now define the matrix-valued function $Q\colon B_\epsilon(\bx)\to \mathbb R^{n\times n}$ by the columns
$$Q(\by)=(\bw_1(\by),\ldots,\bw_n(\by)).$$
Note that 
$Q\in C^2(B_\epsilon(\bx),\mathbb R^{n\times n})$ due to the smoothness of $\bff$ and $M$, and $Q$ is invertible for every $\by$. We have $\bw^TM(\by)\bff(\by)=0$ if and only if $\bw\in \spann (\bw_2(\by),\ldots,\bw_n(\by)$, which in turn is equivalent to 
$\bu\in \spann(\be_2,\ldots,\be_n)=:E_{n-1}$, where $\bu=Q(\by)^{-1}\bw$ and $\be_1,\ldots,\be_n$ denotes the standard basis in $\mathbb R^n$.

Now we write
\begin{eqnarray*}
\lefteqn{L_M(\by)}\\
&=&\max_{\bw^TM(\by)\bw=1,\bw^TM(\by)\bff(\by)=0}
\frac{1}{2}\bw^T\left[M(\by)D\bff(\by)+D\bff(\by)^TM(\by)+M'(\by)\right]\bw\\
&=&
\max_{\bu^TQ(\by)^TM(\by)Q(\by)\bu=1,\bu\in E_{n-1}}
\frac{1}{2}\bu^TQ(\by)^T\\
&&\hspace{1cm}\left[M(\by)D\bff(\by)+D\bff(\by)^TM(\by)+M'(\by)\right]Q(\by)\bu.
\end{eqnarray*}
Denoting by $[A]_{n-1}\in \mathbb S^{n-1}$ the lower-right square $(n-1)$ matrix of $A\in \mathbb S^{n}$ and
% $E:=\left(\begin{array}{cc}I_{n-1}&0\\0&0\\ \end{array}\right)$, we obtain thus
 with $\bu=\left(\begin{array}{l}0\\\btu \end{array}\right)$, where $\btu\in\mathbb R^{n-1}$ we get
\begin{eqnarray*}
L_M(\by)
&=&
\max_{\btu^T[Q(\by)^TM(\by)Q(\by)]_{n-1}\btu=1,\btu\in \mathbb R^{n-1}}
\frac{1}{2}\btu^T\bigg[Q(\by)^T\big[M(\by)D\bff(\by)\\
&&\hspace{1cm}
+D\bff(\by)^TM(\by)+M'(\by)\big]Q(\by)\bigg]_{n-1}\btu.
\end{eqnarray*}
Now denote by $\Chol(A)$ the unique Cholesky decomposition of the symmetric, positive definite matrix $A\in \mathbb S^{n-1}$, such that $\Chol(A)$ is an invertible, upper triangular matrix with $\Chol(A)^T\Chol(A)=A$. Denoting $C(\by):=\Chol([Q(\by)^TM(\by)Q(\by)]_{n-1})\in \mathbb R^{(n-1)\times (n-1)}$ and $\btv=C(\by)\btu\in\mathbb R^{n-1}$ we have
\begin{eqnarray*}
L_M(\by)
&=&
\max_{\|\btv\|=1,\btv\in \mathbb R^{n-1}}
\frac{1}{2}\btv^T(C^{-1}(\by))^T\\
&&\hspace{0.8cm}\left[Q(\by)^T\left[M(\by)D\bff(\by)+D\bff(\by)^TM(\by)+M'(\by)\right]Q(\by)\right]_{n-1}C^{-1}(\by)\btv\\
&=&
\max_{\|\btv\|=1,\btv\in \mathbb R^{n-1}}
\btv^TH(\by)\btv\\
&=&\lambda_{max}(H(\by))
\end{eqnarray*}
where $H(\by)\in \mathbb S^{n-1}$ is defined by 
\begin{eqnarray*}H(\by)&=&\frac{1}{2}(C^{-1}(\by))^T\left[Q(\by)^T\left[M(\by)D\bff(\by)+D\bff(\by)^TM(\by)+M'(\by)\right]Q(\by)\right]_{n-1}\\
&&C^{-1}(\by).
\end{eqnarray*}
%Note that this function is continuous, but not differentiable in general.

The function $\by\to H(\by)$ is continuously differentiable as the Cholesky decomposition, the inverse, the operation $[\cdot]_{n-1}$, $Q$, $M$, $D\bff$ and $M'$ are continuously differentiable by the assumptions. Hence, the function $H(\by)$ is locally Lipschitz-continuous. The function $\lambda_{max}$ is globally Lipschitz-continuous, hence, $L_M$ is locally Lipschitz-continuous.
%
%In particular, $L_M$ is locally Lipschitz-continuous for all $\bx\in \Omega$, and, as $\Omega$ is compact, there is a $\tdelta$ and $L>0$ such that 
% \begin{eqnarray*}
% \left|L_{M}( {\bf x})-L_{M}(\bx)\right|
%&\le&L  \left\| {\bf x}- \bx\right\|
%\end{eqnarray*}
%holds for all $\bx\in \Omega$ and $\|\by-\bx\|<\tdelta$.
%
%We have already shown that 
%$$|\lambda_{max}(N_1)-\lambda_{max}(N_2)|\le \|N_1-N_2\|_2$$
%is globally Lipschitz continuous. It remains to show that 
%$\bx\to A(\bx)$ is Lipschitz-continuous on $\overline{\Omega_\delta}$.
%We will show that the function is $C^1$ and thus it is Lipschitz-continuous on a compact set. Note that each product in the formula is continuously differentiable, as the Cholesky decomposition, the inverse, the operation $[\cdot]_{n-1}$, $Q$, $M$, $D\bff$ and $M'$ are continuously differentiable by the assumptions and definitions of $M$ (note that $\bff\in C^2$ and thus also $M\in C^2$). This shows the statement.
%
%We only need Lipschitz-continuity for $\bx$ and $\by$ with $\by=P(\bx)$ ??? CHECK??? So the line between $\bx$ and $\by$ is in $\overline{\Omega_\delta}$. We need this to apply the mean value theorem implying $C^1$ -> Lipschitz cont, as we need the integral along the connection between $\bx$ and $\by$.
%
%REFORMULATE STATEMENT -- CHECK THAT $C^1$ ARGUMENT REALLY WORKS. WHICH NORMS?? all  matrix norms are equivalent, so perhaps the element-wise???
\end{proof}
\end{appendix}

%\section*{References}
{\small
	\bibliographystyle{plain}
	\bibliography{mybibfile}
}
%
%\begin{thebibliography}{9}
%\bibitem{aylward}
%Aylward, Parrilo, Slotine 2008, Automatica
%\bibitem{forni}
%Forni, Sepulchre
%\bibitem{giesl04}
%Giesl 2004, Necessary conditions for a limit cycle and its basin of attraction, Nonlinear Anal.
%
%\bibitem{giesl07}
%Giesl 2007,
%On the determination of the basin of attraction of a periodic orbit in two-dimensional systems, J. Math. Anal. Appl. 335 (2007) 461--479.
%
%
%\bibitem{giesl07lnm}
%Giesl 2007,
%Lecture Notes in Math.
%
%\bibitem{gieslhafstein13}
%Giesl, Hafstein 2013,
%convex
%
%
%\bibitem{lohmiller}
%Lohmiller, Slotine 1998 On contraction analysis for non-linear systems, Automatica
%\end{thebibliography}

\end{document}